\documentclass[12pt]{article}
\usepackage{latexsym,amssymb,amsmath,amsfonts,amsthm}
\usepackage{tikz}
\newtheorem{theorem}{Theorem}
\newtheorem{lemma}{Lemma}

\newtheorem{corollary}{Corollary}
\theoremstyle{definition}

\def\beq{ \begin{equation} }
\def\eeq{ \end{equation} }
\def\mn{\medskip\noindent}

\def\ep{\varepsilon}

\def\square{\vcenter{\vbox{\hrule height .4pt
  \hbox{\vrule width .4pt height 5pt \kern 5pt
        \vrule width .4pt} \hrule height .4pt}}}

\def\L{\mathbb{L}}

\def\FF{\mathcal{F}}

\def\RR{\mathbb{R}}
\def\ZZ{\mathbb{Z}}
\def\AA{\mathcal{A}}
\def\LL{\mathcal{L}}
\def\QQ{\mathcal{Q}}
\def\HH{\mathcal{H}}
\def\VV{\mathcal{V}}

\def\EE{\mathcal{E}}
\def\BB{\mathcal{B}}

\def\clearp{}
\def\hh{\hspace{1ex}}
\def\tcr#1{#1}

\begin{document}

\title{A stochastic spatial model for the \\
sterile insect control strategy}
\author{Xiangying Huang and Rick Durrett\thanks{Both authors were partially supported by NSF grant DMS 1809967
from the probability program.} \\
Dept.~of Math, Duke University 
}
\date{\today}

\maketitle

\begin{abstract}
In the system we study, 1's and 0's represent occupied and vacant sites in the contact process with births at rate $\lambda$ and deaths at rate 1. $-1$'s are sterile individuals that do not reproduce but appear spontaneously on vacant sites at rate $\alpha$ and die at rate $\theta\alpha$. We show that the system (which is attractive but has no dual) dies out at the critical value and has a nontrivial stationary distribution when it is supercritical. Our most interesting results concern the asymptotics when $\alpha\to 0$. In this regime the process resembles the contact process in a random environment.

\end{abstract}  

\section{Introduction}

The idea that populations of economically important insect species might be controlled, managed or eradicated through genetic manipulation was conceived in the late 1930s by an American entomologist, Dr Edward F. Knipling. The sterile insect control strategy is based on releasing overwhelming numbers of sterile male insects into the wild. The sterile males compete with normal males to mate with the females. Females that mate with a sterile male produce no offspring, thus reducing the next generation's population. Sterile insects are not self-replicating and, therefore, cannot become established in the environment. Repeated release of sterile males when the population density is low can further reduce and in some cases eliminate pest populations. The technique has successfully been used in a large number of situations. See the 2005 book by Klassen and Curtis \cite{KlaCur} and the 2021 second edition of Dyck, Hendricks, and Robinson \cite{DHR}.

Here, we will construct a simple stochastic spatial model for this system. Rather than having male and female flies, the normal and sterile flies will compete for space, which we model as the $d$-dimensional integer lattice $\ZZ^d$. In our process $\xi_t$ each site $x\in \ZZ^d$ can be in state $\{1,0,-1\}$, where state 1 =  normal fly, 0 = empty site, and $-1 =$ sterile fly. We use this ordering of states so that the system will be attractive: if $\xi_0(x) \leq \bar{\xi}_0(x)$ for all $x\in\ZZ^d$ then the two processes can be constructed on the same space so that $\xi_t(x) \leq \bar{\xi}_t(x)$ for all $x\in\ZZ^d$ and all $t\geq 0$. See, for example, (B12) and (B13) in Liggett's book \cite{Liggett99} for a formal definition on attractiveness for interacting particle systems.

The 1's and 0's are an ordinary contact process.  $-1$'s are sterile individuals that do not reproduce but appear spontaneously on vacant sites at rate $\alpha$ and die at rate $\theta\alpha$. Thus, if $N_1$ is the number of nearest neighbors occupied by 1's, then the process has the following transition rates:
\begin{align*}
0\to 1 &\text{ at rate } \lambda N_1\\
1\to 0 &\text{ at rate }1\\
0\to -1&\text{ at rate } \alpha\\
-1\to 0&\text{ at rate }\theta\alpha.
\end{align*}
The death rate of sterile flies is not the same as ordinary flies since they have been treated with radiation or chemicals to make them sterile. We have chosen the last two rates so that we can let $\alpha\to 0$ and consider a situation in which the sterile flies change on a much slower time scale. This is probably not biologically realistic but, as the reader will see, it is mathematically interesting.

After this paper was submitted we learned that in 2016  Kevin Kuoch \cite{Kuoch} constructed a stochastic spatial model for the sterile fly control strategy
in which the evolution of the wild population is governed by a contact process whose growth rate is slowed down in presence of sterile individuals. Each site of $\ZZ^d$ is either empty (state 0), occupied by wild individuals only (state 1), by
sterile individuals only (state 2), or by both wild and sterile individuals (state 3). The rate with which wild individuals give birth (to wild individuals) on neighboring sites is $\lambda_1$ at sites in state 1 and $\lambda_2 < \lambda_1$ in state 3. Deaths of each type of individual are at rate 1. 

Kuoch's paper and ours are part of a large literature on multitype contact process that aims to understand what type of interactions allow for the two types to coexist. The first result in this direction was obtained in 1992 by Neuhauser \cite{Neu92} for the competing contact process in which sites are in state 0 (vacant), 1 or 2, which means they are occupied by one particle of type 1 or type 2. Individually the particles of type $i$ give birth at rate $\beta_i$ and die at rate $\delta_i$. It is conjectured (and partially proved) that if $\beta_1/\delta_1 > \beta_2/\delta_2$ and we start with infinitely many 1's then the 2's become extinct. In ecology this is referred to as ``Gause's Principle", which states that the number of coexisting species is limited by the number of resources. In this example of competing contact process there is one resource, space.

In 2008 Durrett gave Wald Lectures given at the Bernoulli Society meeting in Singapore. The associated paper \cite{RDWald} published in Annals of Applied Probability described results for a number of systems: grass-bushes-trees, colicin, multitype biased voter models, cyclic systems, spatial Prisoner's dilemma, etc. See the paper for details and references. More recent examples include  host-pathogen systems \cite{DL08},  the Staver-Levin model of competition of savanna and forest \cite{DZ15}, and the symbiotic contact process \cite{DY20}. These systems are challenging to study because even if they are attractive they do not have a dual process.

Remenik \cite{Remenik08} has earlier considered a system very similar to our sterile fly model. In his system $\eta_t : \ZZ^d \to \{1,0,-1\}$, where 1 = occupied, 0 = vacant, and $-1 =$ uninhabitable. His rates use different notation ($\beta f_1$ instead of $\lambda N_1$, and $\delta$ instead of $\theta$) but the major difference is that the third line above is changed to
$$
1, 0\to -1 \text{ at rate } \alpha.
$$
This may look like a minor change, but it greatly simplifies the analysis of the process. As Remenik says in his introduction, ``This version is simpler than the alternative in which only 0's can turn into $-1$'s (mainly because our process satisfies a self-duality relationship).'' To explain this, we construct Remenik's process in Section \ref{sec:grep} from a graphical representation, which allows us to define a dual process. Our process $\xi_t$ can be constructed on that structure so that if we have $\xi_0 \ge \eta_0$ then  $\xi_t \ge \eta_t$ for all $t$. We can also construct the contact process $\zeta_t$ on the same space with the other two processes so that if $\zeta_0 \ge \xi_0$ then $\zeta_t \ge \xi_t$ for all $t\ge 0$. Hence we have the following

\begin{theorem}
 If Remenik's process $\eta_t$ survives then our process $\xi_t$ does. If the ordinary contact process $\zeta_t$ dies out then $\xi_t$ does.
\end{theorem}

Combining this observation with Theorem 1 in Remenik's paper \cite{Remenik08} gives information about the phase diagram drawn in Figure  \ref{fig:pd}. To explain the left hand side of the picture, we note that in Remenik's model sites transition to $-1$ at rate $\alpha$ independent of whether the state is 1 or 0, and back to 0 at rate $\alpha\theta$, so the fraction of time spent in state $-1$ is $1/(1+\theta)$. The intervals in which a site is in state $-1$ are exponential with rate $\alpha\theta$. If in the complement of the union of these intervals there are no infinite paths that go up and jump to nearest neighbors, then the system will die out for all $\lambda$. This occurs for small $\theta$.

\begin{figure}[ht]
\begin{center}
\begin{picture}(240,220)
\put(30,30){\line(1,0){180}}
\put(30,30){\line(0,1){180}}
\put(210,30){\line(0,1){180}}
\put(30,210){\line(1,0){180}}
\put(30,80){\line(1,0){180}}
\put(7,75){$\lambda_{con}$}
\put(15,140){$\lambda$}
\put(10,205){$\infty$}
\put(85,50){$\xi_t$ dies out}
\put(28,10){0}
\put(130,10){$\theta$}
\put(205,10){$\infty$}
\put(95,10){$\theta_c$}
\put(100,25){\line(0,1){10}}
\put(100,80){\line(0,1){130}}
\put(40,195){$\lambda_{Rem} = \infty$}
\put(100,210){\line(1,-5){10}}
\put(110,160){\line(1,-3){12}}
\put(122,124){\line(1,-1){20}}
\put(170,90){\line(-2,1){28}}
\put(210,80){\line(-4,1){40}}
\put(140,160){$\xi_t$ survives}
\end{picture}
\caption{Phase diagram when $\alpha=1$. $\lambda_{Rem}$ is the critical value for Remenik's model, while $\lambda_{con}$ is the critical value for the ordinary contact process.}
\end{center}
\label{fig:pd}
\end{figure}
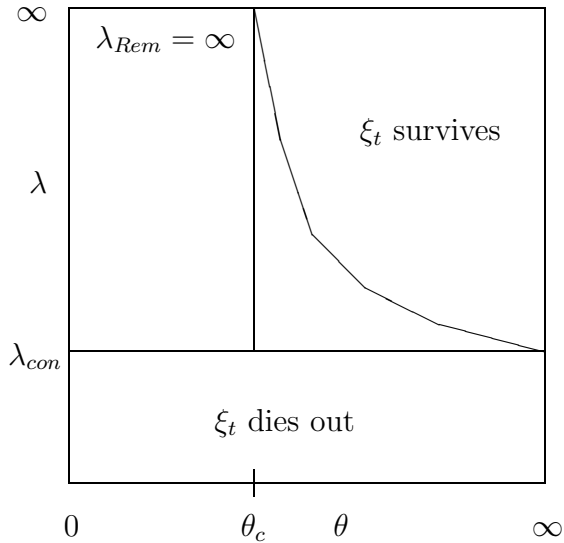

Let $\xi_t^0$ be the process starting with a single 1 at the origin, and, with some abuse of notation, let $\xi_t^0$ also denote the set of 1's in the process. For fixed $\alpha$ and $\theta$ let 
$$
\lambda_c(\alpha,\theta) = \inf \{ \lambda: P( \xi^0_t \neq \emptyset \hbox{ for all $t$}) > 0 \},
$$
be the critical value for survival starting from a single occupied site. 
Using the classical block construction for the contact process introduced in Bezuidenhout and Grimmett \cite{BG90} (see also Section 2, Part I of Liggett's book \cite{Liggett99}), it is straightforward to show Theorem \ref{dieatcrit}. The block construction is a general technique in the study of interacting particle systems where we look at the macroscopic behavior of the process on a large space-time block of our choice, and use that information to derive results on the process. This idea will also be applied in the proof of other theorems in this paper.

\begin{theorem}\label{dieatcrit}
When $\lambda = \lambda_c(\alpha,\theta)$, $\xi_t$ dies out. When $\lambda > \lambda_c(\alpha,\theta)$ there is a nontrivial stationary distribution $\xi^1_\infty$, which is the limiting distribution of the process $\xi_t$ starting from all 1's.
\end{theorem}

Since our process is attractive, if we let $\xi^1_t$ be the process starting from all 1's then $\xi^1_t$ converges to a limit that we call $\xi^1_\infty$ as time goes to infinity. For a contact process $\zeta_t$, by duality we have $P(\zeta^1_\infty(x)=1)=P(\zeta^x_t\neq \emptyset \text{ for all }t>0)>0$ when $\lambda>\lambda_{con}$. Hence the existence of a nontrivial stationary distribution for the contact process is a straightforward consequence of duality. However, our process does not have a dual so we need to show that $\xi^1_\infty$ is nontrivial for $\lambda > \lambda_c(\alpha,\theta)$.

Remenik \cite{Remenik08} proved this conclusion for his model. Since his process has a dual, a corollary of the proof of Theorem \ref{dieatcrit} is the complete convergence theorem. We will state his result in the next section after we have described the graphical representation.

The proofs of Theorems 1 and 2 are fairly routine. Things become more interesting when we study the behavior when $\alpha$ is small. In this case the system behaves like a contact process in a random environment where $-1$'s are obstacles that block the birth of 1's. However, over longer time scales the $-1$'s flip so the situation is different from the models of the contact process in a static random environment that have been studied earlier. Our initial goal was to find the asymptotics for the critical value when $\alpha\to 0$. As you will see, that has turned out to be a very challenging problem.

\subsection{Generic block construction} \label{sec:generic}

The proofs of Theorems  \ref{percsurv}, \ref{d1die}, and \ref{d2die} which describe properties of the system with small $\alpha$ are based on a thirty-year-old technique called the block construction. The first example was \cite{MBRD}. Durrett's 1993 St.~Flour Notes \cite{StFlour} give a description of the theory and a number of examples. Modifications of this technique are needed in the three proofs, so we begin with a description of a generic application: proving that a system $\xi_t: \ZZ \to \{0,1\}$ has a positive probability of not dying out. 

Let ${\cal L} = \{ (m,n) \in \ZZ^2 : m,n \ge 0 \hbox{ and $m+n$ is even} \}$ be the renormalized lattice and turn ${\cal L}$ into an oriented graph by adding edges $(m,n) \to (m+1,n+1)$ and $(m,n) \to (m-1,n+1)$. The first coordinate is space while the second is time. For each $(m,n) \in {\cal L}$ we have a block 
$$
B_{m,n} = (2mL,nT) + [-4L,4L] \times [0,T].
$$
The next picture should help explain the definitions.

\begin{figure}[ht]
\begin{center}
\begin{picture}(260,180)
\put(40,30){\line(1,0){160}}
\put(40,90){\line(1,0){200}}
\put(80,150){\line(1,0){160}}
\put(40,30){\line(0,1){60}}
\put(80,90){\line(0,1){60}}
\put(200,30){\line(0,1){60}}
\put(240,90){\line(0,1){60}}
\put(100,25){\line(0,1){10}}
\put(140,25){\line(0,1){10}}
\put(180,85){\line(0,1){10}}
\put(140,85){\line(0,1){10}}
\put(100,85){\line(0,1){10}}
\put(60,85){\line(0,1){10}}
\put(220,145){\line(0,1){10}}
\put(180,145){\line(0,1){10}}
\put(140,145){\line(0,1){10}}
\put(100,145){\line(0,1){10}}
\put(25,15){$-4L$}
\put(190,15){$4L$}
\put(90,12){$-L$}
\put(135,12){$L$}
\put(45,73){$-3L$}
\put(170,73){$3L$}
\put(233,75){$6L$}
\put(90,158){$-L$}
\put(135,158){$L$}
\put(170,158){$3L$}
\put(210,158){$5L$}
\put(118,28){$\star$}
\put(78,88){$\star$}
\put(158,88){$\star$}
\put(118,148){$\star$}
\put(198,148){$\star$}
\put(117,38){\vector(-2,3){30}}
\put(123,38){\vector(2,3){30}}
\put(157,98){\vector(-2,3){30}}
\put(163,98){\vector(2,3){30}}
\end{picture}
\caption{Picture of the generic block construction in $d=1$. 
Points of ${\cal L}$ are marked by $\star$s.}
\label{genericbc}
\end{center}
\end{figure}
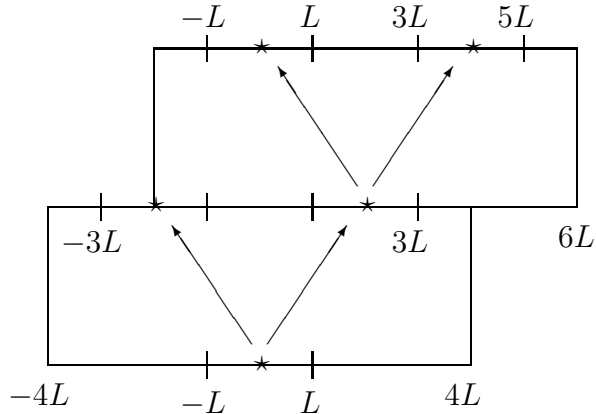

We say that the point $(m,n) \in {\cal L}$ is \textit{wet} if the configuration $\xi_{nT}$ restricted to $I_m = 2mL + [-L,L]$ has a specified property $H$ (for happy), e.g., there are at least $K$ sites in state 1 in $I_m$ at time $nT$.
In most applications the process $\xi_t$  is constructed from an infinite collection of Poisson processes called a graphical representation, see e.g., Section \ref{sec:grep}. We will construct good events $G_{m,n}$ so that a wet site $(m,n)$ can make both $(m+1,n+1)$ and $(m-1,n+1)$ wet if $G_{m,n}$ occurs. Usually $G_{m,n}$ is chosen to have suitable translation invariant properties.

The main step in the block construction is to show that given $\ep>0$ we can pick $L$ and $T$ so that if $(m,n)$ is wet then  $(m-1,n+1)$ and $(m+1,n+1)$ are wet with probability at least $1-\ep$, and the good event $G_{m,n}$ that guarantees this can be determined from the Poisson points in $B_{m,n}$. Note that the box $B_{m,n}$ intersects $B_{m-2,n}$ and  $B_{m+2,n}$ but does not intersect the interior of any other $B_{k,\ell}$ with $(k,\ell) \in {\cal L}$. This observation implies that the events $G_{m,n}$ have a finite range of dependence, $M$, so Theorem 4.1 in \cite{StFlour} guarantees that if $\ep<\ep_M$ for some $\ep_M$ depending on $M$ and $(0,0)$ is wet then with probability at least $19/20$ there is an infinite path of wet sites on ${\cal L}$ starting at $(0,0)$. If $H$ has been defined suitably then the existence of an infinite path implies that the process does not die out. A second result in \cite{StFlour}, Theorem 4.2, allows one, in most cases, to prove the existence of a nontrivial stationary distribution.

\clearp

\subsection{Survival for small $\alpha$ in $d=2$} \label{sec:introth3}

Let $p_c^{site}(\ZZ^2)$ denote the critical value of the site percolation in $\ZZ^2$.  We consider only the two-dimensional case to make the percolation arguments simpler. In this case the theory is developed by the use of ``sponge crossings,'' i.e., left-to-right or top-to-bottom crossings of rectangles by open sites. See, for example, Kesten's book \cite{Kesten}. In $d=3$ a left-right crossing of a cube need not intersect a top-to-bottom crossing, so the theory becomes much more complicated, see Grimmet's book \cite{Grimmett}.

\begin{theorem} \label{percsurv}
In $d=2$ if $\theta/(1+\theta)>p_c^{site}(\ZZ^2)$ and $\lambda > \lambda_{con}(\ZZ)$ the critical value for the contact process on $\ZZ$, then our process survives for small $\alpha$.
\end{theorem} 

There are four main ingredients in the proof.

\begin{figure}[ht]
\begin{center}
\begin{picture}(240,240)
\put(30,30){\line(1,0){180}}
\put(30,85){\line(1,0){180}}
\put(30,95){\line(1,0){180}}
\put(30,145){\line(1,0){180}}
\put(30,155){\line(1,0){180}}
\put(30,210){\line(1,0){180}}
\put(30,30){\line(0,1){180}}
\put(85,30){\line(0,1){180}}
\put(95,30){\line(0,1){180}}
\put(145,30){\line(0,1){180}}
\put(155,30){\line(0,1){180}}
\put(210,30){\line(0,1){180}}
\put(85,218){$Q_1$}
\put(145,218){$Q_3$}
\put(215,147){$Q_2$}
\put(215,87){$Q_4$}
\put(20,15){$-N$}
\put(75,15){$-N/3$}
\put(140,15){$N/3$}
\put(205,15){$N$}
\end{picture}
\caption{Picture of the rectangles in which crossings will occur}
\label{2dbox}
\end{center}
\end{figure}
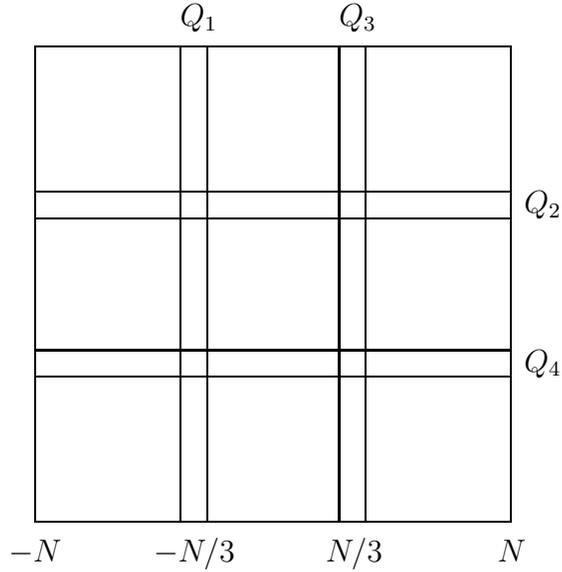

\begin{itemize}
\item
Let $T=\ep_0/\alpha$ for some $\ep_0$ to be chosen later in Section \ref{sec:spweb}. We say that a site $x\in\ZZ^2$ is \textit{closed} if it is in state $-1$ at some time in $[0,2T]$, and \textit{open} otherwise. Consider a large square $[-N,N]^2$. Let $\eta>0$ be small and define four rectangles by
\begin{align}
Q_1 & = [-N/3-N^\eta,-N/3+N^\eta] \times [-N,N],
\nonumber\\
Q_2 & = [-N,N] \times [N/3-N^\eta,N/3+N^\eta], 
\label{fourQ}\\
Q_3 & = [N/3-N^\eta,N/3+N^\eta] \times [-N,N], 
\nonumber\\
Q_4 & = [-N,N] \times [-N/3-N^\eta,-N/3+N^\eta]. 
\nonumber
\end{align}
We use the word \textit{open crossings} to refer to top-to-bottom crossings of $Q_1$ and $Q_3$ and  left-to-right crossings of $Q_2$ and $Q_4$ by paths of open sites. In Section \ref{sec:spweb} we use results about site percolation in two dimensions to show that if sites are independent and open with probability $p>p_c^{site}(\ZZ^2)$ and $N$ is large then with high probability there are open crossings in all $Q_1, Q_2, Q_3, Q_4$.

\item
Our next step is to introduce the block construction. 
Writing NE for north-east we define
$$
{\cal L}^{NE}=\{ (m,n,m+n) \in \ZZ^3:  m, n\geq 0\}
$$
 and draw an oriented edges from $(m,n,m+n) \to(m+1,n,m+n+1)$ and from
 $(m,n,m+n) \to(m,n+1,m+n+1)$. A site $(m,n,m+n)\in\mathcal{L}^{NE}$ has an associated box 
$$
B_{m,n}=((mN,nN)+[-N,N]^2) \times [(m+n-1)_+T, (m+n+1)T].
$$
The notation $(\cdot)_+$ means $\max\{0, \cdot\}$, and it only has an effect when $m=n=0$.

The somewhat unusual geometry of ${\cal L}^{NE}$ is forced on us by the fact that the influence of the initial configuration of $-1$'s persists for time $1/\alpha$ in expectation which is $1/\ep_0$ levels in the construction. Hence, if we pick $\ep_0$ small then the range of dependence in the block construction is large. By moving north or east on each step we can guarantee a finite range of dependence independent of $\ep_0$.

\begin{figure}[ht]
\begin{center}
\begin{picture}(260,170)
\put(30,30){\line(1,0){100}}
\put(30,60){\line(1,0){150}}
\put(80,90){\line(1,0){150}}
\put(130,120){\line(1,0){100}}
\put(180,150){\line(1,0){50}}
%
\put(30,30){\line(0,1){30}}
\put(80,30){\line(0,1){60}}
\put(130,30){\line(0,1){90}}
\put(180,60){\line(0,1){90}}
\put(230,90){\line(0,1){60}}
\put(45,65){$B_{0,0}$}
\put(95,95){$B_{1,0}$}
\put(145,125){$B_{2,0}$}
\put(195,155){$B_{3,0}$}
\put(45,20){wet}
\put(95,20){open}
\put(97,50){wet}
\put(145,50){open}
\put(147,80){wet}
\put(195,80){open}
\put(197,110){wet}											
\end{picture}
\caption{Block construction viewed from the $x$-axis.}
\label{fig:th3block}
\end{center}
\end{figure}
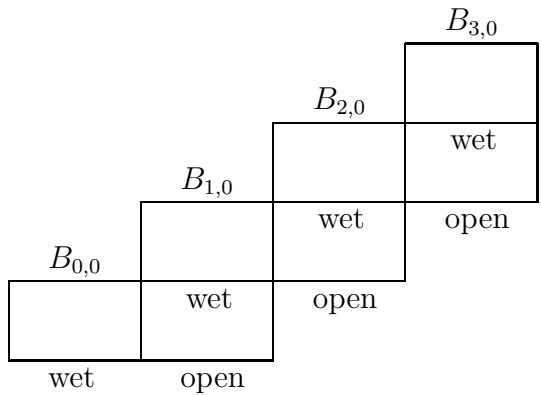

\item
We say that $(m,n,m+n) \in {\cal L}^{NE}$ with $m+n>0$ is {\it open} if there are open crossings of $(mN,nN) +Q_i$ for all $1\leq i\leq 4$ that persist through $[(m+n-1)_+T, (m+n+1)T]$. The block $B_{0,0}$ is excluded here, see Figure \ref{fig:th3block}. An open crossing of $Q_i$ is a self-avoiding path and hence isomorphic to an interval $[0,\ell]$ on $\ZZ$ where $2N \le \ell \le 4N^{1+\eta}$. By the definition of openness for a site there are no $-1$'s on the crossing during $[(m+n-1)_+T,(m+n+1)T]$. 

We say an open crossing of length $\ell\geq 2N$ is $\ep$-good if when we map the configuration of 1's and 0's to the interval $[0,\ell]$ on $\ZZ$ and put 0's outside $[\ep N, \ell-\ep N]$, the contact process survives with probability at least $1-\ep$. We say that $(m,n,m+n) \in {\cal L}^{NE}$ with $m+n\ge 0$ is {\it wet} if, in addition to $(m,n,m+n)$ being open, the open crossings in $Q_2$ and $Q_3$ are $\ep$-good at time $(m+n)T$. In Lemma \ref{transmission} in Section \ref{sec:renorm} we will show that if $(m,n,m+n)$ is wet and $(m+1,n, m+n+1)$ is open then with high probability $(m+1,n, m+n+1)$ will become wet. By symmetry the same conclusion holds with $(m+1,n, m+n+1)$ replaced by $(m,n+1, m+n+1)$. 

\item
The final detail is to specify the initial condition. Let $p_0 \in (p_c^{site}(\ZZ^2),\theta/(\theta+1))$ and let $\zeta^{pr}_0(x)$ be independent with $P(\zeta^{pr}_0(x)=0)=p_0$ and $P(\zeta^{pr}_0(x)=-1)=1-p_0$. If there are no 1's in the initial configuration then the process is doomed to die out, so we flip coins with a small probability of heads to replace a positive fraction of the 0's in $B_{0,0}$ by 1's. 

When $m+n=1$ the condition at time $(m+n-1)T=0$ is a product measure so Corollary \ref{sweb} in Section \ref{sec:spweb} shows that with high probability these sites are open. To check that $(m,n,m+n)\in {\cal L}^{NE}$ is open for $m+n>1$ we note that if the initial condition $\zeta^{pr}_t(x)$ is 0 with probability $p_0$ and $-1$ with probability $1-p_0$,  then the fraction of 0's at time $t/\alpha$ satisfies
$$
f'(t) = - f'(t) + \theta (1-f(t)), \quad f(0)=p_0.
$$
It is easy to see that as $t$ increases from 0 to $\infty$,
$f(t)$ increases from $p_0$ to $\theta/(\theta+1)$. If we introduce 1's into the initial configuration $\zeta^{pr}_0$ to produce $\zeta_0$ then $\zeta_0 \ge \zeta^{pr}_0$ and we have $\zeta_t \ge \zeta^{pr}_t$ for all $t\ge 0$ and then $(m,n,m+n)$ is open with high probability. 

\end{itemize}

Combining the observations above and using the block construction we conclude that if $\alpha$ is small then with positive probability we have 
$$\Omega_\infty=\{ \text{there is an infinite sequence of wet sites in the oriented percolation on ${\cal L}^{NE}$}\}.$$
Since boxes corresponding to wet sites have at least one occupied site, it follows that when $\Omega_\infty$ occurs our process does not die out.  

\clearp

\subsection{Extinction for small $\alpha$ in $d=1$} \label{sec:introth4}

In the other direction one might hope to prove the following

\mn{\bf Naive Conjecture.} {\it If $\theta/(1+\theta)<p_c^{site}(\ZZ^2)$ then for any fixed $\lambda$ the contact process will die out for small $\alpha$. }

\mn
If the random environment was static this would be easy. The contact process evolving on a finite cluster will quickly die out. However, the flipping of sites from $-1$ to 0 will allow the process to move between different finite clusters.

The next result shows that this guess is correct in $d=1$ where $p_c^{site}(\ZZ)=1$. 

\begin{theorem} \label{d1die}
In $d=1$ if $\lambda,\theta$ are fixed then the process dies out almost surely for small $\alpha$.
\end{theorem}

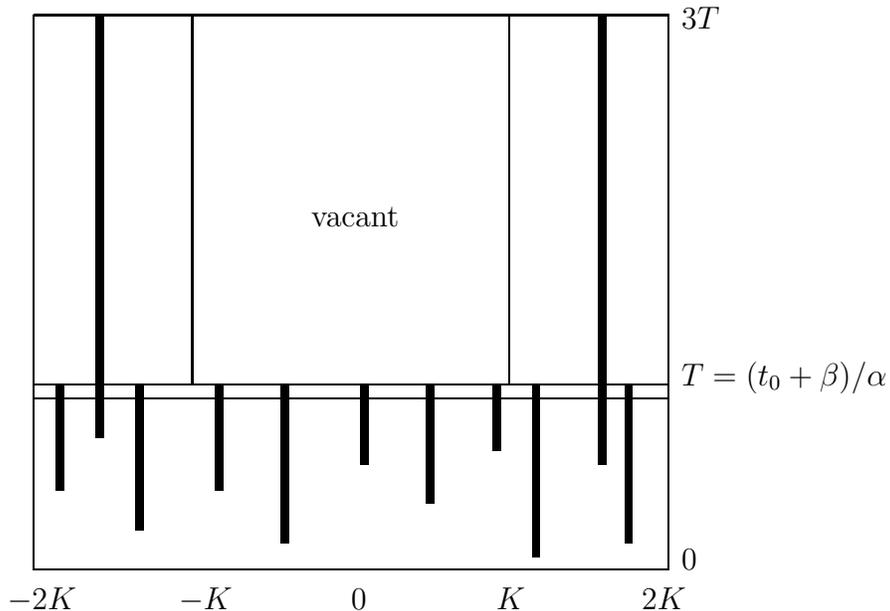
\begin{figure}[ht]
\begin{center}
\begin{picture}(300,260)
\put(30,30){\line(1,0){240}}
\put(30,100){\line(1,0){240}}
\put(30,95){\line(1,0){240}}
\put(30,240){\line(1,0){240}}
\put(30,30){\line(0,1){210}}
\put(270,30){\line(0,1){210}}
\put(90,100){\line(0,1){140}}
\put(210,100){\line(0,1){140}}
\put(20,15){$-2K$}
\put(85,15){$-K$}
\put(150,15){0}
\put(205,15){$K$}
\put(260,15){$2K$}
\put(275,30){0}
\put(275,100){$T = (t_0+\beta)/\alpha$}
\put(275,235){$3T$}
\put(135,160){vacant}
\linethickness{1mm}
\put(40,60){\line(0,1){40}}
\put(55,80){\line(0,1){160}}
\put(70,45){\line(0,1){55}}
\put(100,60){\line(0,1){40}}
\put(125,40){\line(0,1){60}}
\put(155,70){\line(0,1){30}}
\put(180,55){\line(0,1){45}}
\put(205,75){\line(0,1){25}}
\put(220,35){\line(0,1){65}}
\put(245,70){\line(0,1){170}}
\put(255,40){\line(0,1){60}}
\end{picture}
\caption{Picture of the block construction in $d=1$. Dark lines are locations of $-1$ barriers.}
\label{fig:1ddie}
\end{center}
\end{figure}

\noindent
The idea behind the proof, which again is a block construction, can best be conveyed by drawing a picture, see Figure \ref{fig:1ddie}.
Since our process is attractive, we can \tcr{suppose that the system starts in the state of all 1's at time 0, since this is the worst case scenario if we want the process to die out.} There are three phases in time.

\begin{itemize}
 \item  
{\bf Phase 1:} $[0, t_0/\alpha]$. \ When an occupied site becomes vacant, there is a small chance of it changing to $-1$ before it gets occupied again. When this happens the site will stay in state $-1$ for time $\Omega(1/\theta\alpha)$. If $t_0$ is large enough, then at time $t_0/\alpha$ we have space divided into small intervals by $-1$'s. To make sure the space remains a collection of small intervals we employ only the $-1$'s that will not turn back to 0 during time interval $[t_0/\alpha, (t_0+\beta)/\alpha]$. 

 \item 
{\bf Phase 2:} $[t_0/\alpha, (t_0+\beta)/\alpha]$. \ Once $[-2K,2K]$ is broken into a large number of small pieces by $-1$'s, results for the contact process on a finite interval imply that all the 1's in $[-2K,2K]$ will die in this time interval.

\item 
{\bf Phase 3:} $[T, 3T]$. \ To kill off the process using a block construction we want to have $[-K,K]$ remain vacant during time $[T,3T]$ where $T =(t_0+\beta)/\alpha$. In $d=1$ it is enough to build two walls of $-1$'s to protect the middle region from being populated again. Using a comparison with oriented percolation, the vacant regions combine to make strips of width $2K$ in which there are no 1's. As Figure \ref{fig:overlap} shows the vacant regions associated with $(m,n)$ and $(m+1,n+1)$ overlap nicely. 
\end{itemize}

\begin{figure}[ht]
\begin{center}
\begin{picture}(180,220)
\put(30,30){\line(1,0){120}}
\put(30,150){\line(1,0){120}}
\put(30,30){\line(0,1){120}}
\put(150,30){\line(0,1){120}}
\put(90,70){\line(1,0){120}}
\put(90,190){\line(1,0){120}}
\put(90,70){\line(0,1){120}}
\put(210,70){\line(0,1){120}}
\put(-30,70){\line(1,0){120}}
\put(-30,70){\line(0,1){120}}
\put(-30,190){\line(1,0){120}}
\put(20,15){$-2K$}
\put(50,15){$-K$}
\put(90,15){0}
\put(115,15){$K$}
\put(140,15){$2K$}
\put(175,55){$3K$}
\put(205,55){$4K$}
\put(-15,55){$-3K$}
\put(-45,55){$-4K$}
\put(-45,65){$T$}
\put(-50,105){$2T$}
\put(-50,145){$3T$}
\put(-50,185){$4T$}
\linethickness{1mm}
\put(60,70){\line(1,0){60}}
\put(60,70){\line(0,1){80}}
\put(120,70){\line(0,1){120}}
\put(60,150){\line(1,0){60}}
\put(120,110){\line(1,0){60}}
\put(180,110){\line(0,1){80}}
\put(120,190){\line(1,0){60}}
\put(0,110){\line(1,0){60}}
\put(0,110){\line(0,1){80}}
\put(0,190){\line(1,0){60}}
\put(60,110){\line(0,1){80}}
\end{picture}
\caption{Picture of the overlap of the (0,0), (1,1) and $(-1,1)$ boxes. The rectangles with thick lines are vacant if these boxes are good. }
\label{fig:overlap}
\end{center}
\end{figure}
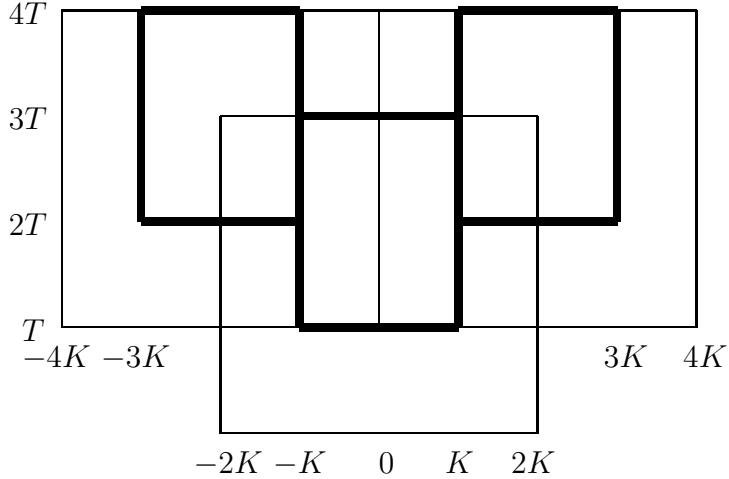

\mn
{\bf Block construction.} We use the generic renormalized lattice 
$$
\LL=\{ (mK,nT): m,n\in\ZZ, m+n \text{ is even}\}.
$$
A \textit{block} at $(m,n)\in \{ (i,j)\in \ZZ^2: i+j \text{ is even}\}$ i
$$
B_{m,n} = (mK,nT)+([-2K,2K] \times [0,3T]). 
$$
The block at (0,0) is \textit{good} if when we start our process $\xi_t$ with all 1's on $[-2K,2K]$ at time 0, there are no 1's in $[-K,K] \times [T,3T]$. This definition is extended to other blocks by translation. A site $(m,n)$ is said to be \textit{open} if the corresponding block is good, otherwise it is said to be \textit{closed}. The block events have a finite range of dependence $M$. We will show in Section \ref{sec:pfth3} that when $\alpha$ is sufficiently small, we can choose $K$ and $T$ so that a site $(m,n)$ is open with probability $>1-\ep_M$.

Define \textit{wet} sites on level $n$ by  $W^0_n=\{ y: (0,0) \to (y,n)\}$, where ``$\to$" means there is a path of open sites connecting the two sites. Let $\ell^0_n = \min W^0_n$ and $r^0_n = \max W^0_n$. Classical oriented percolation results show that (see e.g., Section 3 in \cite{RD84}) there is $v>0$ so that 
$$
\ell^0_n/n \to -v, \quad r^0_n/n \to v \quad \text{ a.s. on }\Omega^0\equiv \{ W^0_n\neq \emptyset \text{ for all }n\}.
$$
It is customary to use $\alpha$ for the ``edge speed.'' However $\alpha$ is one of the parameters of the model so instead we use v (for velocity).

\begin{figure}[ht]
\begin{center}
\begin{picture}(240,220)
\put(20,190){\line(1,0){200}}
\put(90,160){dead zone}
\linethickness{0.5mm}
\put(30,190){\line(0,-1){80}}
\put(50,190){\line(0,-1){100}}
\put(70,170){\line(0,-1){100}}
\put(90,110){\line(0,-1){60}}
\put(110,90){\line(0,-1){60}}
\put(130,130){\line(0,-1){100}}
\put(150,150){\line(0,-1){100}}
\put(170,170){\line(0,-1){100}}
\put(190,190){\line(0,-1){60}}
\put(210,190){\line(0,-1){40}}
\put(30,190){\line(1,0){20}}
\put(30,150){\line(1,0){20}}
\put(30,110){\line(1,0){20}}
\put(50,170){\line(1,0){20}}
\put(50,130){\line(1,0){20}}
\put(50,90){\line(1,0){20}}
\put(70,110){\line(1,0){20}}
\put(70,70){\line(1,0){20}}
\put(90,90){\line(1,0){20}}
\put(90,50){\line(1,0){20}}
\put(110,70){\line(1,0){20}}
\put(110,30){\line(1,0){20}}
\put(130,50){\line(1,0){20}}
\put(130,90){\line(1,0){20}}
\put(130,130){\line(1,0){20}}
\put(150,70){\line(1,0){20}}
\put(150,110){\line(1,0){20}}
\put(150,150){\line(1,0){20}}
\put(170,130){\line(1,0){20}}
\put(170,170){\line(1,0){20}}
\put(190,190){\line(1,0){20}}
\put(190,150){\line(1,0){20}}
\end{picture}
\caption{Picture of the walls and the resulting dead zone, which includes the rectangles that make up the walls. }
\label{fig:overlap}
\end{center}
\end{figure}
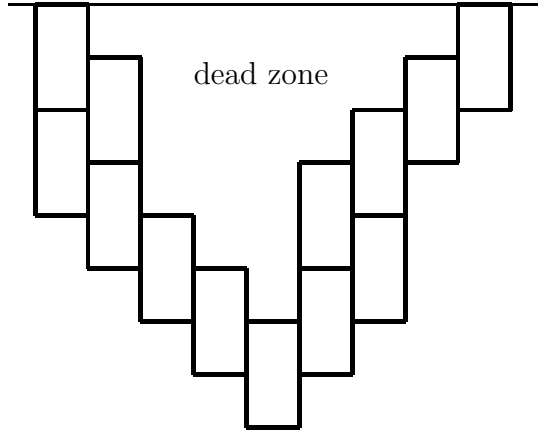

There is a path of open sites from $(0,0)$ to $(\ell^0_n,n)$ and to $(r^0_n,n)$. Due to our construction, these paths of open sites are associated with a
collection of $2K \times 2T$ rectangles that are known to be vacant. We call the union of these rectangles associated with one of the two paths a wall. There cannot be any 1's in between the two walls, since if there were its ancestors must have traveled through the vacant rectangles, which is impossible because the walls contain no 1's. The existence of this growing ``dead zone'' implies that our process $\xi_t$ dies out on the event $\Omega^0$, which by the block construction occurs with a positive probability $\rho$.

To show that the 1's die out we have to upgrade ``with positive probability" to ``with probability 1". To do this we will follow the motto ``If at first you don't succeed, then try again.'' Let $\tau = \min\{ n : W^0_n = \emptyset\}$. The sites on levels $\ge \tau+M+1$ are independent of the failure of the first attempt. So we try the construction again starting at site $m=0$ if the time $\tau+M+1$ is even or $m=1$ if the time is odd. The second attempt has an independent probability $\rho$ of success. Eventually we will have a success which kills off the 1's and the proof is complete.

\clearp

\subsection{Extinction for small $\alpha$ in $d=2$} \label{sec:introth5}

A major problem that prevents us from proving a converse to Theorem \ref{percsurv} is that when 1's are present the density of $-1$'s that we can guarantee is much smaller than $1/(1+\theta)$. which is the density we can guarantee when there are only $-1$'s and $0$'s. In two dimensions our model has the following rates

\begin{center}
\begin{tabular}{clccl}
$1 \to 0$ & at rate 1 & \qquad & $0\to -1$ & at rate $\alpha$  \\
$0 \to 1$ & at rate $\le 4\lambda$ & \qquad & $-1\to 0$ &  at rate $\theta\alpha$ 
\end{tabular}
\end{center}
 
Turning the ``$\le 4\lambda$" into  ``$=4\lambda$" we have a three state birth and death process in which the equilibrium frequencies, which satisfy detailed balance, are
 
\beq\label{pid=2}
\pi(1) = \frac{4\lambda\theta}{1+\theta+4\lambda\theta}, \quad
\pi(0) = \frac{\theta}{1+\theta+4\lambda\theta}, \quad
\pi(-1) = \frac{1}{1+\theta+4\lambda\theta}.
\eeq

Let $p_c^{site}(\ZZ^2)$ be the critical value for the site percolation in $\ZZ^2$ and let $\chi(p)$ denote the mean cluster size when the probability for a site to be open is $p$.

\begin{theorem} \label{d2die}
Let $p_0<p_c^{site}(\ZZ^2)$ be the constant such that 
$$
-4\chi(p_0)^2\log((1-e^{-1})e^{-4\lambda})=1.
$$
If $1-\pi(-1)<p_0$ then our process in $d=2$ dies out for small $\alpha$.
\end{theorem}

The proof again uses a block construction argument.  The ideas are similar to the proof of Theorem \ref{d1die} (compare with the previous bullet list) but some of the details are much different. 

\begin{itemize}
\item
{\bf Phase 1:} $[0,t_0/\alpha]$. If $t_0$ is large enough, $-1$'s produce giant component. To avoid having $-1$'s turn into 0's we employ only $-1$'s that will not turn back to 0 during time interval $[t_0/\alpha, (t_0+\beta)/\alpha]$. 

\item 
{\bf Phase 2:} $[t_0/\alpha, (t_0+\beta)/\alpha]$. If the giant component is sufficiently dense, results for the contact process on a finite set imply that all the 1's in $[-2K,2K]^2$ will die in this time interval.
 
\item
{\bf Phase 3:} $[T,3T]$. \ In $d=1$ it was sufficient to have a two long-lived $-1$'s to prevent the vacant region we create before time $T$ from being reinvaded during $[T,3T]$.  In $d=2$, to prevent reinvasion we make the vacant region that we create so large that it is very unlikely for individuals outside $[-2K,2K]^2$ to reach $[-K,K]^2$ in time $2T$.
\end{itemize}

\mn
{\bf Block construction.} Space is two-dimensional so the renormalized lattice is
$$
{\cal L}^3 = \{ (m,n,k) \in \ZZ^3 : m+n\hbox{ is even }, k \ge 0 \}
$$
A block at $(m,n,k)$ is
$$
B_{m,n,k} = (mK,nK,kT) + ( [-2K,2K]^2 \times [0,3T])
$$
The block at (0,0,0) is \textit{good} if we start with all 1's in $[-2K,2K]^2$ at time 0 then there are no 1's in the space-time box $[-K,K]^2 \times [T,3T]$. This definition is extended to other blocks by translation. A site $(m,n,k)$ is said to be \textit{open} if the corresponding block is good, otherwise it is said to be \textit{closed}.

Define \textit{wet} sites on level $n$ by  
$$
W^0_n=\{ (x,y) : (0,0,0) \to (x,y,n)\}
$$
where ``$\to$" means there is a path of open sites connecting the two sites. For comparison we define the process starting with all sites wet at time 0,
$$
\bar W_n=\{ (x,y): (x_0,y_0,0)\to (x,y,n) \text{ for some }(x_0,y_0)\in \ZZ^2\}.
$$

In $d=1$ the fact that the left edge of the wet region $\ell^0_n/n  \to -v$ and the right edge $r^0_n/n \to v$ creates  a linearly growing ``dead zone" that in combination with a block construction guarantees that the process dies out. 

To extend this result to $d=2$ we use Lemma 5.2 from \cite{CDP13} which holds for an $M$-dependent percolation with $\ep < \ep_M$.

\begin{theorem} [\textbf{Shape theorem}] There is a convex set $D$ so that on the event $\Omega_\infty = \{ W^0_n \neq \emptyset \hbox{ for all $n$} \}$, for any $\eta>0$,

\mn
(i) There exists some $n_0(\eta)\in \mathbb{N}$ so that $W^0_n \subseteq (1+\eta)n D$ for all $n \ge n_0(\eta)$.

\mn
(ii) On $W^0_n \cap (1-\eta)n D$ the process $W^0_n$ is equal to  $\bar W_n$  with high probability.
\end{theorem}

\noindent
Here and throughout the paper ``with high probability" means with a probability that tends to 1 as $n\to\infty$.

As we will explain in Section \ref{sec:compop} arguments in Chapter 5 of \cite{CDP13} can be used to show that if $\ep$ is small enough then  the density of wet sites is large enough. Thus with high probability there will be no particles in the set $n(1-2\eta)D$ and we will have a  linearly growing dead zone. Once this is established the proof can be completed as in $d=1$.

\subsection{Organization of the paper}

The remainder of the paper is devoted to proofs. In Section 2 we describe the graphical representation of Remenik's model, the resulting duality, and prove Theorem 1. In Section 3 we prove Theorem 2 by using Liggett's version of the Bezuidenhout-Grimmett argument. In Sections 4, 5, and 6 we prove Theorems 3, 4, and 5. The block constructions have been described in the introduction so it only remains to show that the parameters can be chosen so that the block events occur with high probability.

\clearp 

\section{Graphical representation} \label{sec:grep}

The graphical representation allows us to construct our process from a collection of independent Poisson processes in a way that processes with different parameters can be coupled together through their graphical representations. It is a very useful tool in the study of interacting particle systems. For details the reader is referred to Section III.6 of \cite{Liggett99}. We state our construction as following:
\begin{itemize}
 \item For each ordered pair $(x,y)$ of nearest neighbors in $\ZZ^d$, let $\{T^{x,y}_n, n\ge 1\}$ be a Poisson process with rate $\lambda$. At each arrival we draw an arrow from $x$ to $y$ to indicate that $y$ will change to 1 if $x$ is in state 1 and $y$ is state 0. 
\item For every $x$, let $\{T^x_n, n \ge 1\}$ be a Poisson process with rate $1$. At each arrival we write a $\bullet_1$ to indicate that a 1 at the site will turn to 0.
\item For every $x$, let $\{T^{\alpha,x}_n, n \ge 1\}$ be a Poisson process with rate $\alpha$. At each arrival we place a $\bullet_{-1}$ symbol to indicate that the site will become $-1$ if the current state is $0$.
\item For every $x$, let $\{T^{\theta,x}_n,n \ge 1\}$ be a Poisson process with rate $\alpha\theta$. At each arrival we place a $\ast_{-1}$ symbol to indicate that if the state is $-1$ it will return to state 0. 
\end{itemize}

\noindent
The same graphical representation can be used to construct Remenik's process if we change the third rule to: $\bullet_{-1}$ indicates that the site will change to state $-1$ in spite of its current state. The contact process can be constructed on the same space by ignoring the last two collections of Poisson processes.

\subsection{Duality for Remenik's model}

\noindent
Recall that Remenik's process $\eta_t$ has the following transition rates
\begin{align*}
0\to 1 &\text{ at rate } \lambda N_1\\
1\to 0 &\text{ at rate }1\\
1,0\to -1&\text{ at rate } \alpha\\
-1\to 0&\text{ at rate }\theta\alpha
\end{align*}
On each interval between a $\bullet_{-1}$ and a $\ast_{-1}$, we know that the system is in state $-1$. At all other times the state is 1 or 0. To identify sites occupied by 1's we say there is an \textit{active path} up from  $(x,s)$ to $(y,t)$ if there is a path that only moves up, crosses arrows in the direction of their orientation and avoids $\bullet_1$'s and sites that have been set equal to $-1$. Let $A_t = \{ y: $ there is an active path up from $(x,0) \to (y,t)$  for some $x\in A_0\}$ where $A_0$ is the set of 1's at time 0. In the same way that we defined active paths up then we can define active paths down. They only move down, cross arrows and the direction opposite their orientation and avoids $\bullet_1$'s and sites that have been set equal to $-1$. Let $\mu_\rho$ be the product measure of 0's and $-1$'s, in which $-1$'s have probability $\rho=1/(1+\theta)$. Let $B_t = \{ x : \eta_t(x)=-1\}$ and set the initial distribution of $B_0$ to be $\mu_\rho$ in order to have a useful duality. Observe that $\mu_\rho$ is the equilibrium for $B_t$.

The dual process $(\hat \eta^t_s)_{0\leq s\leq t}=(\hat A^t_s, \hat B^t_s)_{0\leq s\leq t}$ is constructed using the same graphical representation we used for constructing $\eta_t$. To state the duality let $C \subseteq \ZZ^d$, and define the probability measure $\nu_C$ as following: $-1$'s are first chosen according to the equilibrium $\mu_\rho$, and then for every site in $C$ that is not $-1$ we set its state to be 1. The dual process started with $\nu_C$ will be denoted by $(\hat \eta^{\nu_C,t}_s)_{0\leq s\leq t}$.

Fix $t>0$ and we have the environment process $(B_s)_{0\leq s\leq t}$. The dual environment is given by $\hat B^t_s =B_{t-s}$ for $0\leq s\leq t$. 
Placing a 1 at time $t$ at  every site in $C\backslash \hat B^t_0$ gives the initial condition $\hat A^{\nu_C,t}_0$.
We define the dual process by 
$$
\hat A^{\nu_C,t}_s = \{ y: \text{there is an active path from $ (y,t-s)\to (x,t)$  for some $x \in \hat A^{\nu_C,t}_0$} \}
$$ 
Let $A, C, D$ be finite subsets of $\ZZ^d$. Proposition 2.2 in \cite{Remenik08} gives
\beq
P^{\nu_A}( A_t \cap C \neq \emptyset, B_t \cap D \neq \emptyset ) =  
P^{\nu_C}( \hat A^t_t \cap A \neq \emptyset, 
\hat B_0^t \cap D \neq \emptyset ),
\label{asdual}
\eeq
which is the natural generalization of the duality relationship of additive processes
$$
P( \xi^A_t \cap B \neq \emptyset) = P( \tilde \xi^B_t \cap A \neq \emptyset
$$
to processes where the state is described by two sets. In \eqref{asdual} the two sets are equal due to the construction on the graphical prespresentation. However, this immediately implies the following self-duality relation which does not depend on the construction: if $A$ or $C$ is finite, then
\beq
P^{\nu_A}( A_t \cap C \neq \emptyset, B_t \cap D \neq \emptyset ) =  
P^{\nu_C}( A_t \cap A \neq \emptyset, B_0 \cap D \neq \emptyset ).
\label{distdual}
\eeq

Let $\nu_\emptyset$ be the distribution corresponding to having the $-1$'s at equilibrium and no 1's. Let $\eta^1_\infty$ be the limiting distribution when we start the process $\eta_t$ from all 1's. Using duality Remenik was able to prove

\mn
{\bf Complete convergence theorem.} {\it  Denote by $\tau(\eta) = \inf\{ t: \{x: \eta_t(x)=1\} = \emptyset \}$ the extinction time of the process. Then for every initial distribution $\mu$
$$
\eta^{\mu}_t \Rightarrow P^\mu( \tau(\eta) < \infty ) \nu_\emptyset + P^\mu(\tau(\eta) = \infty) \eta^1_\infty.
$$}

\subsection{Properties for our model}

\mn
\textbf{Attractiveness.}
As stated in Section 1 our process $\xi_t$ is attractive in the sense that if $\xi_0\leq \bar{\xi}_0$ in terms of the partial order $-1\leq 0\leq 1$ then $\xi_t \leq \bar{\xi}_t$ for all $t\geq 0$.

\smallskip
\mn
\textbf{Positive correlations.}
We state a version of positive correlation for our process $\xi_t$. Let $\chi_{A}$ denote the probability measure that assigns mass 1 to the configuration $\xi$ with $\xi|_A\equiv1,\xi|_{A^c}\equiv -1$. Let $f,g$ be increasing real-valued functions depending on finitely many coordinates. Then 
\beq\label{posicor}
E^{\chi_A}fg\geq E^{\chi_A}f \cdot E^{\chi_A}g
\eeq

Since $f$ and $g$ depend on finitely many coordinates and every jump in our process is between states which are comparable in the partial order $-1\leq 0\leq 1$, \eqref{posicor} follows from a result of Harris (see Theorem II.2.14 in Liggett \cite{Lig85}).

\clearp

\section{Proof of Theorem \ref{dieatcrit}}

The proof of Theorem \ref{dieatcrit} follows from a classical block construction argument developed in \cite{BG90}, which is also covered in great detail in Section 2, Part I of Liggett's book \cite{Liggett99}. Here we present our proof closely following the same organization as that in \cite{Liggett99}. To avoid repetition, we will focus on the proofs that are different from \cite{Liggett99} while omitting the analogous ones.

\mn
\textbf{The boundary of a big box has many infected sites.}
The following lemma is an analogue of Proposition I.2.1 in \cite{Liggett99}. However, the proof is much different and hence is given here.
\begin{lemma}\label{bigbox}
Suppose under our choice of $\lambda,\theta$ and $\alpha$, the process $\xi_t$ survives. Then
$$\lim_{n\to\infty}P^{\chi_{[-n,n]^d}}(\xi_t\neq \emptyset\hh \forall t\geq 0)=1.$$
\end{lemma}

\begin{proof}
For $x\in\ZZ^d$, define the shift transformation $T_x$ by
$$(T_x\eta)(y)=\eta(y-x)$$
where $\eta\in \{-1,0,1\}^{\ZZ^d}$. We start by showing $T_x$ is ergodic. Let $\mathcal{S}$ be the set of events depending only on the Poisson processes on a finite number of sites and edges. Set 
$$\mathcal{A}=\{ A\in\FF: \inf_{B\in\mathcal{S}}P(A\Delta B)=0\}$$
where $\FF$ is the $\sigma$-algebra generated by the Poisson processes on all sites and edges on $\ZZ^d$.

We can easily check that $\mathcal{A}$ is a $\sigma$-algebra. As $\mathcal{S}\subseteq \mathcal{A}$, it follows that 
$\FF=\sigma(\mathcal{S})\subseteq \mathcal{A}$. Hence $\mathcal{A}=\FF$. Let $A\in\FF$ be an event that is invariant under transformation $T_x$, i.e., $A=T_x^{-1}A$. Since $A\in\mathcal{A}$, for any $\ep>0$ there exists some $A_\ep\in\mathcal{S}$ such that $P(A\Delta A_{\ep})\leq \ep$. As $A_\ep$ depends on a finite number of sites and edges, there exists a positive number $M_\ep$ such that $A_\ep \in \FF_{M_\ep}=\sigma(T^{x,y},T^{x},T^{\alpha,x},T^{\theta,x}: x,y\in[-M_\ep,M_\ep]^d, x\sim y).$ For any $C,D,C',D'\subseteq \ZZ^d$, observe that $(C\cap D)\Delta(C'\cap D')\subseteq (C\Delta C')\cup (D\Delta D')$. Hence we have
\begin{align*}
|P(A)^2-P(A)|
&=|P(A)P((T^{-1}_x)^mA)-P(A\cap (T^{-1}_x)^mA)|\\
&\leq |P(A_\ep)P((T^{-1}_x )^mA_{\ep})-P(A_\ep\cap (T^{-1}_x)^mA_\ep)|+4\ep.
\end{align*}
When $m$ is sufficiently large, the first term on the right hand side is 0. Hence $P(A)^2=P(A)$, which implies $P(A)\in\{0,1\}$. Therefore for every $x\in\ZZ^d$, $T_x$ is an ergodic transformation.

Consider $T_{e_1}$ and let $Y_x$ be the indicator of the event $\{\xi^{\chi_x}_t\neq \emptyset \hh\forall t\geq 0\}$, where $x\in\ZZ^d$. Birkhoff's ergodic theorem gives 
$$\frac{1}{(2n+1)} \sum_{x\in[-n,n]\times \{0\}^{d-1}} Y_x=\frac{1}{(2n+1)}\sum_{k=-n}^n (T_{e_1})^k Y_0\to EY_0\quad \text{ a.s. as }n\to\infty.$$
It follows from attractiveness that
\begin{align*}
P^{\chi_{[-n,n]^d}}(\xi_t\neq \emptyset\hh \forall t\geq 0) &\geq P(\sum_{x\in[-n,n]\times \{0\}^{d-1}} Y_x \geq 1)\\
&= P\left(\frac{1}{(2n+1)}\sum_{x\in[-n,n]\times \{0\}^{d-1}} Y_x \geq \frac{1}{2n+1}\right)\to 1 \hh \text{ as }n\to\infty.
\end{align*}
\end{proof}

For $L\geq 1$, let ${}_L\xi_t$ be the truncated process of $\xi_t$ where no births are allowed outside of $(-L,L)^d$. The following Lemma is a straightforward adaptation of Proposition I.2.2 of Liggett \cite{Liggett99}. To be self-contained and succinct, we present a sketch of the proof here. The reader is referred to Proposition I.2.2 of Liggett \cite{Liggett99} for a complete argument.
\begin{lemma}\label{spacetime}
For every finite set $A$ and every $N\geq1$
$$\lim_{t\to\infty}\lim_{L\to\infty} P^{\chi_A}(|_L\xi_t|\geq N)=P^{\chi_A}(\xi_t\neq \emptyset \hh \forall t\geq 0).$$
\end{lemma}

\begin{proof}
Since $\lim_{L\to\infty} P^{\chi_A}(|_L\xi_t|\geq N)=P^{\chi_A}(|\xi_t|\geq N)$, it suffices to show 
$$\lim_{t\to\infty} |\xi_t|=\infty \hh \text{ a.s. on } \{ \xi_s \neq \emptyset \hh \forall s\geq 0\}.$$
Observe that
$$P(\xi_t=\emptyset \text{ for some }t|\FF_s)\geq \left(\frac{1}{1+2d\lambda|\xi_s|}\right)^{|\xi_s|}$$
since each individual has at most $2d$ vacant neighboring sites. 
By the martingale convergence theorem,
$$P(\xi_t=\emptyset \text{ for some }t|\FF_s)\to 1_{\{ \xi_t=\emptyset \text{ for some }t\}} \quad a.s.$$
as $s\to \infty$, which implies the desired result.
\end{proof}

Let 
$$S(L,T)=\{ (x,s)\in\ZZ^d\times [0,T] : \max_i|x_i|=L\}$$
and let $N(L,T)$ be the maximal number of points in a subset of $S(L,T)\cap {}_L\xi$ with the property that any two points $(x,s_1)$ and $(x,s_2)$ in this set satisfies $|s_1-s_2|\geq 1$. The following lemma relates the quantity $N(L,T)$ to $|_L\xi_T|$ using the positive correlations property of our process. The proof of Lemma \ref{relation} is essentially the same as that of Proposition I.2.8 in \cite{Liggett99} and hence is omitted here.
\begin{lemma}\label{relation}
Suppose $L_j\uparrow \infty$ and $T_j\uparrow \infty$. For any $M,N$ and any finite $A\subset \ZZ^d$,
$$\limsup_{j\to\infty} P^{\chi_A}(N(L_j,T_j)\leq M)P^{\chi_A}(|_{L_j}\xi_{T_j}|\leq N)\leq P^{\chi_A}(\xi_t=\emptyset \text{ for some }t).$$
\end{lemma}

Define $N_+(L,T)$ to be the maximal number of space-time points in 
$$S_+(L,T)=\{ (x,s)\in \{L\}\times [0,L)^{d-1}\times [0,T]: x\in{}_L\xi_s\}$$
such that each pair of these points having the same spatial coordinate have their time coordinates at distance at least 1. We have the following by the positive correlation stated in \eqref{posicor}:
\begin{lemma}\label{poscor}
$$
P^{\chi_{[-n,n]^d}}(|_L\xi_T\cap [0,L)^d|\leq N)\leq [P^{\chi_{[-n,n]^d}}(|_L\xi_T|\leq 2^dN)]^{2^{-d}}
$$
and 
$$
P^{\chi_{-[n,n]^d}}(N_+(L,T)\leq M)\leq [P^{\chi_{[-n,n]^d}}(N(L,T)\leq d2^dM)]^{2^{-d}/d}.
$$
\end{lemma}
The reader is referred to Proposition 2.6 and Proposition 2.11 in \cite{Liggett99} for analogues of Lemma \ref{poscor}. The proof of the following theorem is similar to that of Theorem 2.12 in \cite{Liggett99}. To avoid repetition we will present the major steps here and refer the reader to \cite{Liggett99} for details of the proof.

\mn
\textbf{The finite space-time condition.}
\begin{theorem} \label{opcomp}
If $\xi_t$ survives, then it satisfies the following condition: \\For every $\ep>0$ there are choices of $n,L,T$ so that 
\beq\label{result1}
P^{\chi_{[-n,n]^d}}\left({}_{L+2n} \xi_{T+1} \supset x+[-n,n]^d \text{ for some } x\in[0,L)^d\right)>1-\ep
\eeq
and
\begin{align}\label{result2}
\nonumber P^{\chi_{[-n,n]^d}}\big({}_{L+2n} \xi_{t+1} \supset & x+[-n,n]^d \text{ for some } 0\leq t\leq T, \\
& \text{ and for some } x\in \{L+n\}\times [0,L)^{d-1}\big)>1-\ep.
\end{align}
\end{theorem}

\begin{proof}
Given $\delta>0$, by Lemma \ref{bigbox} we can choose a large enough $n$ such that 
\begin{equation}\label{box}
P^{\chi_{[-n,n]^d}}(\xi_t\neq \emptyset\hh \forall t\geq 0)>1-\delta^2.
\end{equation}
Next, Lemma \ref{spacetime} allows us to choose $L_j\uparrow \infty$ and $T_j\uparrow\infty$ so that
\begin{equation}\label{topsites}
P^{\chi_{[-n,n]^d}}(|_{L_j}\xi_{T_j}|>2^dN)=1-\delta
\end{equation}
for each $j\geq 1$. Applying Lemma \ref{relation} with $M$ and $N$ replaced by $Md2^d$ and $N2^d$ respectively, and combined with (\ref{box}) and (\ref{topsites}), there exists some $j$ so that
$$P^{\chi_{[-n,n]^d}}(N(L_j,T_j)>Md2^d)>1-\delta.$$
Let $L=L_j$ and $T=T_j$ for this choice of $j$. By Lemma \ref{poscor} we have
$$P^{\chi_{[-n,n]^d}}(|_L\xi_T\cap [0,L)^d|\leq N)\leq [P^{\chi_{[-n,n]^d}}(|_L\xi_T|\leq 2^dN)]^{2^{-d}}\leq \delta^{2^{-d}}$$
and 
$$P^{\chi_{[-n,n]^d}}(N_+(L,T)\leq M)\leq [P^{\chi_{[-n,n]^d}}(N(L,T)\leq d2^dM)]^{2^{-d}/d}\leq \delta^{2^{-d}/d}.$$

Choose $N$ so large that any $N$ points in $\ZZ^d$ will contain a subset of at least $N'$ points, each pair of which is separated by an $L_\infty$ distance of at least $2n+1$, where $N'$ is chosen so large that 
$$\left( 1- P^{\chi_0}(_n \xi_1 \supseteq [-n,n]^d)\right)^{N'} \leq \delta.$$
Note that $P^{\chi_0}(_n \xi_1 \supseteq [-n,n]^d)$ is positive because the event that all sites in $[-n,n]^d\backslash \{0\}$ first flip to $0$ and then become infected involves only finitely many sites. Hence 
$$P^{\chi_{[-n,n]^d}}\left(_{L+2n} \xi_{T+1} \supseteq x+[-n,n]^d \text{ for some } x\in[0,L)^d\right)\geq(1-\delta^{2^{-d}})(1-\delta).$$
Choosing $\delta$ sufficiently small with respect to $\ep$ gives \eqref{result1}. To show \eqref{result2}, we choose $M$ in a similar fashion so that any $M$ points in $\ZZ^d$ contain a subset of at least $M'$, which is chosen so large that 
$$\left(1-P^{\chi_0}( _n\xi_1 \supseteq [0,2n]\times [-n,n]^{d-1})\right)^{M'} \leq \delta.$$

\end{proof}

\mn
\textit{Proof of Theorem \ref{dieatcrit}.}
Theorem \ref{opcomp} implies that if our process $\xi_t$ survives then it dominates a supercritical two dimensional oriented site percolation in which sites are open with probability $1-\ep$, with sites that can be reached from the origin implying the existence of occupied copies of $[-n,n]^d$ in corresponding regions in space time $\ZZ^d \times [0,\infty)$. See Theorem I.2.23 in  \cite{Liggett99} for a detailed argument on establishing the  comparison. 

The facts that (i) our process dies out at $\lambda_c(\alpha,\theta)$ and (ii) there is a nontrivial stationary distribution when $\lambda>\lambda_c(\alpha,\theta)$ are both consequences of this comparison with oriented percolation (see page 54-56 in \cite{Liggett99} for reference). Since $\ep$ is arbitrarily small, this shows that if our process survives for a parameter $\lambda$ then there is a $\lambda'<\lambda$ for which the process survives, so if survival occurred at the critical value then we would have a contradiction, proving the first part of Theorem \ref{dieatcrit}. The comparison with supercritical oriented percolation also implies the existence of a nontrivial stationary distribution, i.e., $P(\xi^1_\infty(0)=1)>0$.  See \cite{StFlour} for many applications of this idea. So the second statement in Theorem \ref{dieatcrit} follows as well.

\clearp 

\section{Proof of Theorem \ref{percsurv}} \label{sec:pfth3}

\subsection{Four crossings} \label{sec:spweb}

In the site percolation we use a \textit{crossing} to refer to a path made of open sites. In \eqref{fourQ} we defined four thin rectangles $Q_1,\dots,Q_4$ in $[-N,N]^2$. 
Our goal is to show

\begin{lemma} \label{fourcross}
Consider site percolation in which sites are open with probability $p > p_c^{site}(\ZZ^2)$. If $N$ is large then with high probability there are  top-to-bottom crossings of $Q_1$ and $Q_3$ and left-to-right crossings of $Q_2$ and $Q_4$.
\end{lemma}

In order to do this we need some results on site percolation. We use Grimmett's book \cite{Grimmett} for our main reference. Consider the site percolation on $\L=\ZZ^2$ where each site is open independently with probability $p$. Let $\mathcal{C}_0$ denote the cluster of open sites containing the origin and $p_c^{site}:=\sup \{p \in [0, 1]: E_p (|\mathcal{C}_0|) < \infty\}$. We follow Sykes and Essam \cite{Sykes} to define $\L^*$, the matching lattice of $\L$, to be $\ZZ^2$ with all diagonally adjacent vertices connected.

\begin{figure}[ht]
\begin{center}
\begin{picture}(240,220)
\put(30,190){\line(1,0){160}}
\put(30,150){\line(1,0){160}}
\put(30,110){\line(1,0){160}}
\put(30,70){\line(1,0){160}}
\put(30,30){\line(1,0){160}}
\put(190,30){\line(0,1){160}}
\put(150,30){\line(0,1){160}}
\put(110,30){\line(0,1){160}}
\put(70,30){\line(0,1){160}}
\put(30,30){\line(0,1){160}}
\put(190,30){\line(-1,1){160}}
\put(150,30){\line(-1,1){120}}
\put(110,30){\line(-1,1){80}}
\put(70,30){\line(-1,1){40}}
\put(30,150){\line(1,1){40}}
\put(30,110){\line(1,1){80}}
\put(30,70){\line(1,1){120}}
\put(150,30){\line(1,1){40}}
\put(110,30){\line(1,1){80}}
\put(70,30){\line(1,1){120}}
\put(30,30){\line(1,1){160}}
\put(150,190){\line(1,-1){40}}
\put(110,190){\line(1,-1){80}}
\put(70,190){\line(1,-1){120}}
\end{picture}
\caption{Picture of the matching lattice $\L^*$}
\label{dualL}
\end{center}
\end{figure}
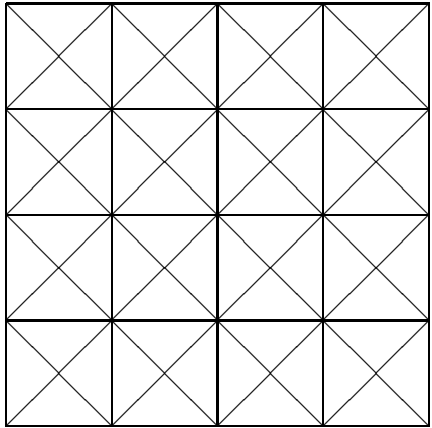

They obtained the relation
\beq\label{critical}
p^{site}_c(\L)+p^{site}_c(\L^*)=1,
\eeq
which was later rigorously proved by Van den Berg \cite{Vdberg}.
See also Chapter 3 of Kesten's 1982 book on percolation, \cite{Kesten}.

The matching pair $\L$ and $\L^*$ share the same set of vertices. A vertex $v\in \L$ is paired with the same vertex $v^*\in \L^*$. If $v$ is open (resp.~closed), then $v^*$ is closed (resp.~open). Thus a dual percolation is defined on $\L^*$ where each site is open with probability $1-p$. When $p>p^{site}_c(\L)$, according to (\ref{critical}) the dual percolation on $\L^*$ is subcritical. We use $P^*_{1-p}$ to denote the law of the dual percolation.

It is well known that for subcritical percolation on $\ZZ^2$, the size of the cluster $\mathcal{C}_0$ has an exponential tail (see for example Section 6.3 in \cite{Grimmett}). Theorem 3 in Antunovi\'{c}  and Veseli\'{c}  \cite{Veselic} proves this result for \textit{quasi-transitive} graphs. A graph $G$ is called \textit{quasi-transitive}, if there exists a finite set of vertices $\FF$ such that for any vertex $x\in G$ there is a $y\in \FF$ and graph automorphism $\varphi \in Aut(G)$ such that $\varphi(y)=x$. 

\begin{lemma}[Theorem 3 \cite{Veselic}]\label{subcri_cluster}
Let $G$ be a quasi-transitive graph and let $p<p^{site}_c(G)= \sup \{p \in [0, 1]: E_p (|\mathcal{C}_0|) < \infty\}$. There is a constant $\gamma_p>0$ such that for any positive integer $n$ we have
$$
P_p(|\mathcal{C}_0| \geq n) \leq \exp(-\gamma_p n).
$$
\end{lemma}

Clearly $\L^*$ is quasi-transitive, so this conclusion holds for percolation on $\L^*$. For the dual percolation on $\L^*$ if $p>p^{site}_c(\L)$ there is a  $\gamma^*_p$ that depends on $p$ such that
\beq\label{exptail}
P^*_{1-p}( |\mathcal{C}_0|\geq n) \leq \exp(-\gamma^*_p n).
\eeq

\begin{proof}[Proof of Lemma \ref{fourcross}]
There is no  top-to-bottom crossing of $Q_1$ on $\L$ if (and only if) there is a left-to-right  dual crossing of $Q_1$ on $\L^*$. There are $2N+1$ possible starting points for left-to-right dual crossings of $Q_1$ on $\L^*$. Since a left-to-right dual crossing has path length at least $2N^\eta$, by \eqref{exptail} the probability of having no top-to-bottom crossing of $Q_1$ on $\L$ is at most $(2N+1)\exp(-2\gamma^*_p N^\eta)$.

Repeating the same argument for $Q_2,Q_3,Q_4$ and using a union bound shows the probability that there are  top-to-bottom crossings of $Q_1$ and $Q_3$ and left-to-right crossings of $Q_2$ and $Q_4$ is at least $1-4(2N+1)\exp(-2\gamma^*_p N^\eta))$.
\end{proof}


Turning back to our process $\xi_t$, recall that a site is open if it never became $-1$ during $[0,2T]$, where $T=\ep_0/\alpha$. Let $p_0 \in (p_c^{site}(\ZZ^2),\theta/(\theta+1))$ and let $\zeta^{pr}_0(x)$ be independent with $P(\zeta^{pr}_0(x)=0)=p_0$ and $P(\zeta^{pr}_0(x)=-1)=1-p_0$. If there are no 1's in the initial configuration then the process is doomed to die out, so we flip coins with a small probability of heads to replace a positive fraction of the 0's in $B_{0,0}$ by 1's. We will show that for a suitably chosen $\ep_0$, the site $(0,0,0)$ is open with high probability. Consequently, by the discussion in Section \ref{sec:introth3} the sites $(m,n,m+n)$ for $m,n\geq 0$ are open with high probability too.

\begin{corollary}\label{sweb}
Suppose $\theta/(1+\theta)>p_c^{site}(\ZZ^2)$. Let $p_0\in  (p_c^{site}(\ZZ^2),\theta/(\theta+1))$ and $T=\ep_0/\alpha$ where $\ep_0$ is a small constant such that $p_0-p_c^{site}(\ZZ^2)>1-e^{-2\ep_0}$. Starting with the product measure $\zeta^{pr}_0$, the site $(0,0,0)\in {\cal L}^{NE}$ is open with arbitrarily large probability when $N$ is sufficiently large.
\end{corollary}

\begin{proof}
A site is closed either if it is initially in state $-1$  or if it flips to $-1$ within time $2T$. Given that we start with a product measure where a site is $-1$ with probability $1-p_0$, the probability that a site is closed is at most $q=1-p_0+1-e^{-2\ep_0}$. By the choice of $\ep_0$ we have $p=1-q>p_c^{site}(\ZZ^2)$ and applying Lemma \ref{fourcross} completes the proof.
\end{proof}

\clearp

\subsection{Contact process on a finite set} \label{sec:cpfs}

By definition, the sites on our four crossings never become $-1$ during time $[0,2T]$, so our model restricted to a crossing behaves exactly like a contact process. Since the crossings are self-avoiding paths they are isomorphic to an interval on $\ZZ$. This leads us naturally to study the behavior of the contact process on a finite set. To do this we will use a construction introduced in Section 9 of \cite{RD84} to study oriented percolation, which is essentially the contact process in discrete time.

The first step in doing this is to go back to Durret's 1980 paper on the one dimensional contact process, \cite{RD80}, which defined an edge speed for the one dimensional contact process. Let $\zeta^A_t$ denote the contact process on $\ZZ$ starting from the set $A$ being occupied and 
$$l^A_t=\inf\{ y: \zeta^A_t(y)=1\},\quad \quad r^A_t=\sup\{ y: \zeta^A_t(y)=1\}$$
denote the left and right edge of the process.

\begin{lemma}\label{edgespeed}
If $\lambda > \lambda_{con}(\ZZ)$ then there is a velocity $v(\lambda)>0$ so that
 as $t\to\infty$, 
$$
r_t^{(-\infty,0]}/t \to v(\lambda) \quad\hbox{almost surely}.
$$
\end{lemma}

Theorem 4 in \cite{RD1983} gives a large deviations result. This and most of the other results we cite here are proved for oriented percolation in \cite{RD84} and can be generalized to the contact process with similar arguments. For this one see Section 11 of \cite{RD84}.

\begin{lemma}\label{thm4}
If $\lambda >  \lambda_{con}(\ZZ)$ then for any $a<v(\lambda)<b$ and $t \ge 0$,
\begin{align*}
P(r^{(-\infty,0]}_t\leq at) & \leq C_0e^{-\gamma_0 t}, \\
P(r^{(-\infty,0]}_t\geq bt) & \leq C_1e^{-\gamma_1 t}.
\end{align*}
where $C_0,\gamma_0$ depend on $a$ and  $C_1,\gamma_1$ depend on $b$.
\end{lemma}

The final ingredient is

\begin{lemma} \label{tauA}
Let $\tau_A = \inf\{ t : \zeta^A_t \equiv 0 \}$. If $\lambda >  \lambda_{con}(\ZZ)$ there is a $\gamma_2>0$ so that
$$
P( \tau^A < \infty ) \le \exp(-\gamma_2 |A|).
$$
\end{lemma}

\mn
See Section 10 of \cite{RD84} for a proof  for oriented percolation. The construction in Section 9 of \cite{RD84} compares the contact process with $M$-dependent oriented site percolation on ${\cal L}$. See Figure 7 in the paper. 

For the contact process we say there exists an \textit{open path} from $(x,s)$ to $(y,t)$, i.e., $(x,s)\to (y,t)$, if there is a path in the graphical representation leading from $(x,s)$ to $(y,t)$ that only moves up, crosses arrows in the direction of their orientation and avoids $\bullet_1$'s.
The first step is to prove that we can get open paths in long thin parallelograms where the sides have slope $1/v(\lambda)$. To do this we will follow the approach in Cristali, Junge, and Durrett \cite{CJD} who studied an inhomogeneous oriented percolation. The parallelogram $\Gamma$ (see Figure \ref{renorm}) has vertices
\begin{align*}
u_0 = (-1.5\delta L,0), \qquad& u_1 = ((1+1.5\delta)L, (1+3\delta)L/v(\lambda)), \\
v_0 = (-0.5\delta L,0), \qquad& v_1 = ((1+2.5\delta)L, (1+3\delta)L/v(\lambda)). 
\end{align*}

In addition we have two intervals $[x_0,y_0]$ and $[x_1,y_1]$ where
\begin{align*}
x_0 = (-1.1\delta L,0), \qquad& x_1 = ((1+1.75\delta)L, (1+3\delta)L/v(\lambda)), \\
y_0 = (-0.9\delta L,0), \qquad& y_1 = ((1+2.25\delta)L, (1+3\delta)L/v(\lambda)). 
\end{align*}

\begin{figure}[ht]
\begin{center}
\begin{picture}(300,220)
\put(30,30){\line(1,1){160}}
\put(110,30){\line(1,1){160}}
\put(30,30){\line(1,0){80}}
\put(190,190){\line(1,0){80}}
\put(185,90){$\leftarrow$ slope $1/v(\lambda)$}
\put(285,185){$t_L$}
\put(130,25){$0$}
\put(25,20){$u_0$}
\put(105,20){$v_0$}
\put(62,25){\line(0,1){10}}
\put(78,25){\line(0,1){10}}
\put(56,15){$x_0$}
\put(74,15){$y_0$}
\put(185,195){$u_1$}
\put(265,195){$v_1$}
\put(205,185){\line(0,1){10}}
\put(245,185){\line(0,1){10}}
\put(200,200){$x_1$}
\put(240,200){$y_1$}
\end{picture}
\caption{The parallelogram $\Gamma$.}
\label{renorm}
\end{center}
\end{figure}
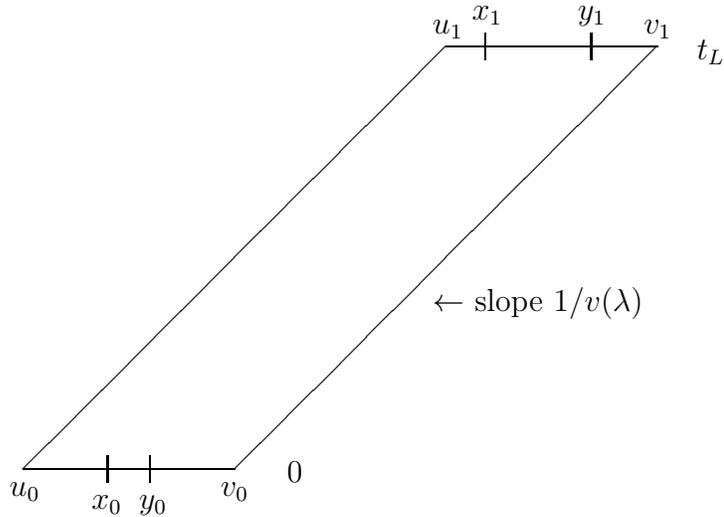

\mn

\begin{lemma} \label{renpath}
Suppose $\lambda > \lambda_{con}(\ZZ)$ and $\delta>0$ is a small constant. There exist finite constants $ \gamma_3, C_3>0$ so that the probability that there is an open path from $[x_0,y_0]\times\{0\}$ to $[x_1,y_1]\times\{ (1+3\delta)L/v(\lambda)\}$ that lies in $\Gamma$ in the contact process is at least $ 1 - C_3 \exp(-\gamma_3 L)$.
\end{lemma}
\begin{proof} By Lemma \ref{tauA} we have
$$
P( \tau^{[x_0,y_0]} < \infty) \le \exp(-\gamma_2(0.2 \delta L)).
$$
That is, with probability at least $1-\exp(-\gamma_2(0.2 \delta L))$ there is a path from $[x_0,y_0]\times\{0\}$ up to time $t_L=(1+3\delta)L/v(\lambda)$. 

The lines from $(x_0,0)$ to $(x_1,t_L)$ and 
from $(y_0,0)$ to $(y_1,t_L)$ have slopes
$$
\frac{(1+2.86\delta)}{(1+3\delta) v(\lambda)}
\quad\hbox{and}\quad 
\frac{(1+3.15\delta)}{(1+3\delta) v(\lambda)},
$$
so Lemma \ref{thm4} implies that the probability for a path starting in $[x_0,y_0]$ to end outside of $[x_1,y_1]$ is $ \le C\exp(-\gamma L)$ for some $C,\gamma>0$.

The final detail is to show that the path cannot escape from $\Gamma$. Intuitively this is true since if the path hits the left edge of $\Gamma$ at some point then it has to travel at an average speed larger than $ v(\lambda)$ later to end up to the right of $ x_1$. The error probabilities depend on $t$ so we first have to rule out hitting the left edge of $\Gamma$ at a time close to $t_L$, say, between time $t_L-(x_1-u_1)/2\lambda$ and $t_L$.  Let $E_\ell(s,t)$ be the event that a path goes from the left edge of $\Gamma$ between time $[s,t]$ to the right of $x_1$ at time $t_L$.

Let $t_1=(x_1-u_1)/2\lambda$. For $E_\ell(t_L-t_1,t_L)$ to occur, the open path has to make at least $x_1-u_1=0.25\delta L$ jumps to the right within $t_1$ units of time. The fastest way for this to happen is $x_1-u_1$ consecutive jumps, each taking an i.i.d. random time $e_i$ following an exponential distribution with rate $\lambda$. Then large deviation implies that
$$P(E_\ell(t_L-t_1,t_L))\leq C'\exp( -\gamma' \delta L)$$
for some $C',\gamma'>0$. 

Let $a_0$ denote the slope of the line from $(u_0,0)$ to $(x_1,t_L)$ and note that $a_0<1/v(\lambda)$. For the event $E_\ell(0,t_L-t_1)$ to occur, there must be an open path starting from the left side of $\Gamma$ with time duration $t\geq t_L-t_1$ that satisfies $r_t\geq t/a_0$. Hence by Lemma \ref{thm4},
$$
P(E_\ell(0,t_L-t_1)) \le  C_0(a_0)\exp(- \gamma_0(a_0)(t_L-t_1)).
$$
Since $P(E_\ell(0,t_L))\leq P(E_\ell(0,t_L-t_1))+P(E_\ell(t_L-t_1,t_L))$ we have obtained an upper bound on $P(E_\ell(0,t_L))$. Touching the right side of $\Gamma$ can be handled in the same way. Combining all the error probabilities completes the proof.

%
\end{proof}

\subsection{Renormalized lattice} \label{sec:renorm}

Let $(c_m,d_n) = (mL,n(1+\delta)L/v(\lambda))$ be the vertices of the renormalized lattice, where $L$ depends on $N$ and will be specified later.  Let $\Gamma_{m,n} = (c_m,d_n) + \Gamma$, $\hat\Gamma = - \Gamma$,
and  $\hat\Gamma_{m,n} = (c_m,d_n) + \hat\Gamma$. The $\Gamma_{m,n}$ will be called \textit{tubes}. The tube $\Gamma_{0,0}$  is said to be \textit{open} if it contains an open path in the contact process described in Lemma \ref{renpath}. The openness of $\Gamma_{m,n}$ is defined via translation. 

The first coordinates of $u_m$ and $v_m$ are
$$
u^0_m =mL - 1.5\delta L \quad\hbox{and}\quad v^0_m = mL - 0.5\delta L
$$
so cutting $\Gamma_{m,n}$ at height $d_n + (1+\delta)L/v(\lambda)$ gives an interval
$$
[(m+1)L -0.5\delta L,(m+1)L + 0.5\delta L],
$$
so it fits right between $\Gamma_{m+1,n+1}$ and  $\hat\Gamma_{m+1,n+1}$, see Figure \ref{fig:block}. Thus a crossing in $\Gamma_{m,n}$ intersects a crossing in $\hat\Gamma_{m+1,n+1}$ which intersects a crossing in $\Gamma_{m+1,n+1}$. Let the box associated with $(c_m,d_n)$ be
$$
B_{m,n} = [mL - (1+3\delta)L, mL + (1+3\delta)L]
\times [d_n, d_n+ (1+3\delta)L/v(\lambda)].
$$

\begin{figure}[h]
\begin{center}
\begin{picture}(240,260)
\put(95,20){$u_m$}
\put(112,20){$v_m$}
\put(110,30){\line(1,2){80}}
\put(120,30){\line(1,2){80}}
\put(110,30){\line(1,0){10}}
\put(150,77){$\Gamma_{m,n}$}
\put(265,27){$d_n$}
\put(190,190){\line(1,0){10}}
\put(215,187){$d_n + (1+3\delta)L/v(\lambda)$}
\put(130,30){\line(-1,2){80}}
\put(140,30){\line(-1,2){80}}
\put(130,30){\line(1,0){10}}
\put(10,30){\line(1,2){80}}
\put(20,30){\line(1,2){80}}
\put(10,30){\line(1,0){10}}
\put(30,30){\line(-1,2){30}}
\put(40,30){\line(-1,2){30}}
\put(30,30){\line(1,0){10}}
\put(230,30){\line(-1,2){80}}
\put(240,30){\line(-1,2){80}}
\put(230,30){\line(1,0){10}}
\put(210,30){\line(1,2){30}}
\put(220,30){\line(1,2){30}}
\put(210,30){\line(1,0){10}}
\put(22,28){$\star$}
\put(122,28){$\star$}
\put(222,28){$\star$}
\put(172,148){$\star$}
\put(72,148){$\star$}
\put(110,145){$B_{m,n}$}
\put(215,147){$d_n + (1+\delta)L/v(\lambda)$}
\put(80,150){\line(-1,2){40}}
\put(90,150){\line(-1,2){40}}
\put(80,150){\line(1,0){10}}
\put(80,150){\line(1,0){10}}
\put(160,150){\line(1,0){10}}
\put(160,150){\line(1,2){40}}
\put(170,150){\line(1,2){40}}
\put(180,150){\line(-1,2){40}}
\put(190,150){\line(-1,2){40}}
\put(180,150){\line(1,0){10}}
\put(60,150){\line(1,0){10}}
\put(60,150){\line(1,2){40}}
\put(70,150){\line(1,2){40}}
\put(125,240){$\hat\Gamma_{m+1,n+1}$}
\put(195,240){$\Gamma_{m+1,n+1}$}
\linethickness{0.5mm}
\put(50,190){\line(1,0){150}}
\put(50,190){\line(0,-1){160}}
\put(50,30){\line(1,0){150}}
\put(200,30){\line(0,1){160}}
\end{picture}
\caption{Picture of the block construction. Stars mark points of the renormalized lattice. }
\label{fig:block}
\end{center}
\end{figure}
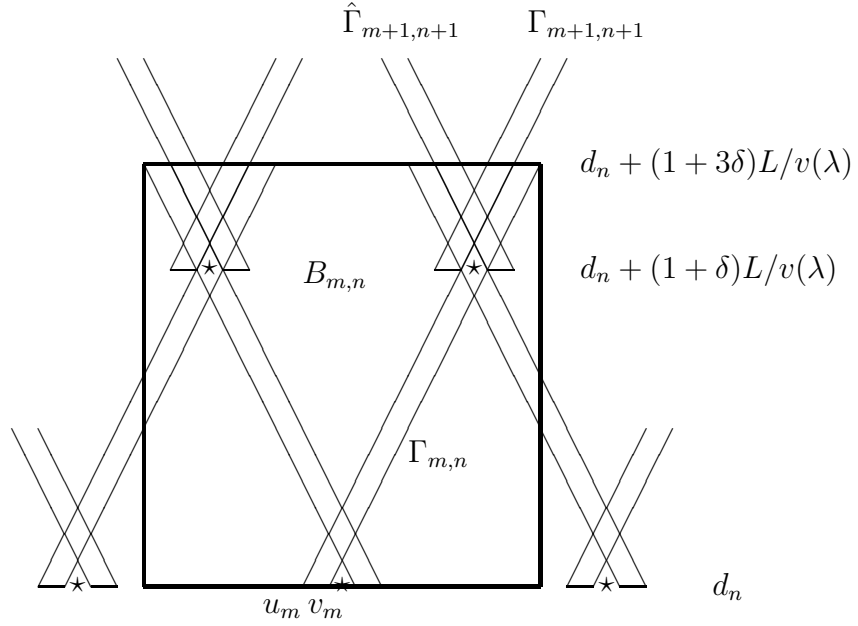

An initial configuration of 1's and 0's on a crossing of length $\ell\ge 2N$ is said to be \textit{$\ep$-good} if when we map this configuration to the interval $[0,\ell]$ on $\ZZ$ and put 0's outside $[\ep N, \ell-\ep N]$ the contact process survives with probability at least $1-\ep$. 
Our next goal is to prove

\begin{lemma}\label{transmission}
Given $\ep>0$, for any small $\alpha$ if we pick $N$ large then
starting with an $\ep$-good configuration on a vertical crossing $\sigma^1$ in $Q_1$, there will be an $\ep$-good configuration on a horizontal crossing $\sigma^2$ in $Q_2$ and an $\ep$-good configuration on a vertical crossing $\sigma^3$ in $Q_3$ at time $T=\ep_0/\alpha$ with probability $\geq 1-2\varepsilon$.
\end{lemma}

\begin{figure}
\begin{center}
\begin{tikzpicture}[thick]
\draw[fill=lightgray] (1,5) rectangle (2,1);
\draw[gray] (0,0) grid (6,6);
\draw[black,ultra thick] (0,1) rectangle (6,2);
\draw[black,ultra thick] (0,4) rectangle (6,5);
\draw[black,ultra thick] (1,6) rectangle (2,0);
\draw[black,ultra thick] (4,6) rectangle (5,0);
\draw[variable=\t, domain=0:6, smooth]  plot ({\t},{0.2*sin(\t r)+4.5});
\draw[variable=\t, domain=1.2:1.8, smooth] plot ({\t},{3*sin(5*(\t-1.5) r)+3});
\draw[variable=\t, domain=4.4:5, smooth] plot ({\t},{-3*sin(5*(\t-4.7) r)+3});
\draw(1.8,6) node[above]{$\sigma^1$};
\draw(0,4.5) node[left]{$\sigma^2$};
\draw(4.5,6) node[above]{$\sigma^3$};
\end{tikzpicture}
\end{center}
\caption{The set-up for Lemma \ref{transmission}. To make the picture easier to draw the rectangles are much wider than $2N^\eta$.}
\label{fig:sponge}
\end{figure}
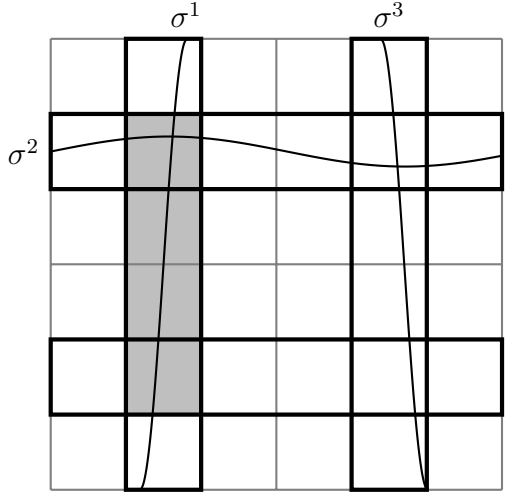

\begin{proof}
An open crossing is a self-avoiding path and hence isomorphic to an interval on $\ZZ$. We assume that $\sigma_1$ maps to $[-n_1,n_1]$ (if the length is even drop one site from the end), where $N \le n_1 \le 2N^{1+\eta}$.

To control the contact process on $[-n_1,n_1]$ we will embed the renormalized lattice into the space-time box 
$$
{\cal S}_1 = [-n_1,n_1] \times [0,N^\rho].
$$ 
$\rho$ will be chosen later, but at this point it is sufficient that $1<\rho<\infty$. We will use the renormalized sites with $c_m \in [-n_1+3\delta L,n_1-3\delta L]$ and $d_n \le N^{\rho}-(1+3\delta)L/v(\lambda)$. In addition we will only use the $\Gamma_{m,n}$ and  $\hat\Gamma_{m,n}$ that are contained in  ${\cal S}_1$. First we will show that with high probability ALL of the tubes in ${\cal S}_1$ in our renormalized lattice are open.  We will call the resulting structure a {\it chain link fence.} Choose $L = N^{\kappa}$ for some $\kappa\in (0,1)$. The dimension of ${\cal S}_1$ implies that there are at most $O(N^{\rho+1+\eta-\kappa})$ tubes in ${\cal S}_1$.  Applying Lemma \ref{renpath} and a simple union bound shows that we will have a chain link fence in ${\cal S}_1$ with probability at least $1-O(N^{\rho+1+\eta-\kappa})\cdot C_3 \exp(-\gamma_ 3N^\kappa)$.

The chain link fence represents open paths in the graphical representation of the contact process. Hence if the contact process survives by time $(1+3\delta)L/v(\lambda)$, it must have hit an open path in the chain link fence and will spread over time through the chain link fence. Starting from an $\ep$-good configuration on $\sigma^1$, the rightmost occupied site is at distance at least $\ep N$ to the end point of $\sigma^1$. Lemma \ref{thm4} with our choice of $L=N^{\kappa}$ implies that with high probability the right edge of the contact process $\zeta^1_t$ on $[-n_1,n_1]$ (i.e., on $\sigma^1$) will not reach $n_1$ by time $(1+3\delta)L/v(\lambda)$. The same argument applies for the left edge. That is, we can view the process $\zeta^1_t$ as the contact process on $\ZZ$ up to time $(1+3\delta)L/v(\lambda)$. It follows from the definition of an $\ep$-good configuration that $\zeta^1_t$ survives up to time $(1+3\delta)L/v(\lambda)$ with probability at least $1-\ep$. From then on the contact process will spread through the chain link fence so that for $n\in \mathbb{N}$ such that $n\geq 2n_1/L$ and $d_n \le N^{\rho}-(1+3\delta)L/v(\lambda)$, we will have at least one occupied site (on the intersection points of the chain link fence) near each $(c_m,d_n)$ with $c_m \in [-n_1+3\delta L,n_1-3\delta L]$.

The next step is to spread the contact process from $\sigma_1$ in $Q_1$ to $\sigma_2$ in $Q_2$. Embed $\sigma_2$ into $\ZZ$ as $[0,n_2]$ with the left end point of $\sigma^2$ sent to 0. Let $m_2$ be the right-most point on $\sigma_2$ that is in $\sigma_1$. We do this so that there are no points of $\Gamma_1$ in $(m_2,n_2]$. We will use the renormalized sites with $c_m \in [m_2+3\delta L,n_2-3\delta L]$ and $d_n \le N^{\rho}-(1+3\delta)L/v(\lambda)$. We will only use the $\Gamma_{m,n}$ and  $\hat\Gamma_{m,n}$ that are contained in  ${\cal S}_2=[m_2,n_2]\times [0,N^\rho]$. 

The argument here is very much the same as that for $Q_1$. Following the same argument we can prove the existence of a chain link fence in ${\cal S}_2$ with high probability. Next we observe that if the contact process on $\sigma^2$ survives through an interval $I_n = [d_n,d_n + 2(1+3\delta)L/v(\lambda)]$ then its open path must have hit the chain link fence in ${\cal S}_2$. Once the contact process spreads through the chain link fence, there will be at least one occupied site near each $(c_m,d_n)$ with $c_m \in [m_2+3\delta L,n_2-3\delta L]$ and $d_n \le N^{\rho}-(1+3\delta)L/v(\lambda)$. By Lemma \ref{tauA} this is more than enough to conclude that the configuration on $\sigma^2$ is $\ep$-good at time $T=\ep_0/\alpha$ when $\alpha$ is sufficiently small and $N$ is large.

Now it remains to show the contact process on $\sigma^2$ survives through some $I_n$ with high probability. Note that for $n\geq 2n_1/L$, in every interval $I_n$, $m_2$ will be occupied at least twice. Each time when $m_2$ is occupied, we can start an independent contact process on $\sigma^2$ with $m_2$ initially occupied. For $n\in \mathbb{N}$ satisfying $n\geq 2n_1/L$ and $d_n \le N^{\rho}-(1+3\delta)L/v(\lambda)$ we can try sufficiently many times for one of the contact processes to survive on $\sigma^2$ through some interval $I_n$. 

The proof for spreading the contact process from $\sigma^2$ to $\sigma^3$ is the same as the last paragraph. Choosing $N$ and $\alpha$ suitably we can make the error probability arbitrarily small and hence the proof of Lemma \ref{transmission} is complete.

\end{proof}

At this point we have verified all of the claims in the sketch of the proof given in Section \ref{sec:introth3} and the proof is complete.

\clearp 

\section{Proof of Theorem \ref{d1die}}\label{sec:pfth4}

We will follow the approach described in the introduction. Let $T=(t_0+\beta)/\alpha$ for some $t_0,\beta>0$ and let $K$ be a large integer. The values of $t_0$, $\beta$ and $K$ will be determined later in the proof.

\begin{lemma}\label{d1perc}
For any $\varepsilon>0$, there exists a choice of $K$ and $T$ so that 
$$
P(\text{$(0,0)$ is open})\geq 1-\varepsilon
$$
when $\alpha$ is sufficiently small.
\end{lemma}

\noindent 
By translation invariance the result holds for any site $(m,n) \in {\cal L}$. 

\mn
\textit{Proof of Lemma \ref{d1perc}.}
Recall that $T=(t_0+\beta)/\alpha$. We consider the following three phases: $[0,t_0/\alpha], [t_0/\alpha, T]$ and  $[T, 3T]$.

\mn
{\bf Phase 1: $[0, t_0/\alpha]$.} In this phase we will create a positive density of $-1$'s. 

\mn
Consider a comparison process $\hat{\xi}_t$ with the following transition rates:
\begin{center}
\begin{tabular}{clccl}
$1 \to 0$ & at rate 1 & \qquad & $0\to -1$ & at rate $\alpha$  \\
$0 \to 1$ & at rate $ 2\lambda$ & \qquad & $-1\to 0$ &  at rate $\theta\alpha$ 
\end{tabular}
\end{center}
In the original process $\xi_t$, the transition rates are the same except that 0 turns to 1 at rate $\lambda N_1$, where $N_1$ is the number of occupied neighbors. In $d=1$, this rate is always $\leq 2\lambda$. Hence if $\hat\xi_0(x)=1$ then $P(\hat{\xi}_t(x)=-1) \le  P(\xi_t(x)=-1)$. Note that in the comparison process the states of different sites are independent.

By a straightforward calculation the equilibrium density for
the birth and death process  $\hat{\xi}_t$, which satisfies detailed balance, is:
$$
\pi(1)=\frac{2\lambda \theta}{1+\theta+2\lambda \theta}, \quad \pi(0)=\frac{\theta}{1+\theta+2\lambda \theta}, \quad \pi(-1)=\frac{1}{1+\theta+2\lambda \theta}.
$$
In our initial condition $\xi_0(x) \equiv 1$ so we take $\hat\xi_0(x) \equiv 1$,
Let $\rho(t) = P(\hat{\xi}_t(x)=-1)$. Markov chain theory implies $\rho(t) \to \pi(-1)/2$. In order to specify a value for $t_0$ we will ound the rate of convergence. Let $X_t, Y_t$ be two Markov chains with the above transition rates with $X_0=1$ and $Y_0$ following the equilibrium distribution $\pi$. Let $\mu_t$ represent the distribution of $X_t$ and let $\tau=\inf \{ t: X_t=Y_t\}$. Then a standard coupling argument of  the two Markov chains $X_t$ and $Y_t$,
see e.g., Section 5.6 in \cite{PTE5}, implies that 
$$
|\rho(t)-\pi(-1)|\leq ||\mu_t-\pi||_{TV}\leq P(\tau>t).
$$
Let $p_{x,y}$ be the probability that the Markov chains starting from $x$ and $y$ respectively will hit by time $1$. It is easy to see that $\min_{x,y\in \{-1,0,1\}} p_{x,y}\geq c_0\alpha$ for some $c_0>0$ by writing out the probabilities. It then follows that for any integer $m$
$$
P(\tau>m)\leq P(\hbox{binomial}(m, c\alpha)=0)
=(1-c_0\alpha)^m \leq e^{-c_0\alpha m}.
$$ 
This implies that we can choose $t_0$ large so that 
$$
|\rho(t_0/\alpha)-\pi(-1)|\leq 2P(\tau>t_0/\alpha)\leq 2e^{-Ct_0}\leq \pi(-1)/2,
$$
i.e., $P(\xi_{t_0/\alpha}(x)=1)\geq \rho(t_0/\alpha)\geq \pi(-1)/2$.

\mn
{\bf Phase 2: $[t_0/\alpha,(t_0+\beta)/\alpha]$.} Starting with the $-1$'s at time $t_0/\alpha$ dominating a product measure with density at least $ \pi(-1)/2$, we will with high probability kill all 1's in $[-2K,2K]$.

\mn
Choose $\beta$ iso that $e^{-\theta \beta}>3/4$. This implies that the probability a site is in state $-1$ from time $t_0/\alpha$ up to time $T=(t_0+\beta)/\alpha$ is at least $\nu =3\pi(-1)/8$. We will call these sites ``walls''. Suppose that the distance between two consecutive walls is $m$. There is a probability at least
$$
[(1-e^{-1})e^{-2\lambda}]^m
$$
to kill all the particles in the interval by time 1: all sites are hit by deaths and there are no births. The probability we fail to do this in $N$ tries is at most
\beq\label{fail1}
( 1-  [(1-e^{-1})e^{-2\lambda}]^m )^N.
\eeq
For all $t\geq 0$ the set of $-1$'s in $\xi_t$ dominates the set of $-1$'s in $\hat\xi_t$, which is a product measure with density $\nu$, the distance between two consecutive walls is bounded by a Geometric($\nu$) random variable. Let 
$$
M(K)=\frac{-\log(4K^2)}{\log(1-\nu)}
$$
be chosen so that the probability of a gap of size $>M(K)$ is 
$$
\le (1-\nu)^{M(K)} = (4K^2)^{-1}. 
$$
Note that there can be at most $2K$ gaps between walls in $[-2K,2K]$. Let $A_K$ be the event that there are two consecutive walls separated by distance at least $M(K)$ in $[-2K,2K]$. Then we have
$$
P(A_K) \leq 2K(1-\nu)^{M(K)}=\frac{1}{2K} \to 0 
$$
as $K\to \infty$. Let $G_2$ denote the event that all the 1's in $[-2K,2K]$ die during this phase. It follows from \eqref{fail1} that 
\beq\label{failphase2}
P(G_2^c)\leq P(A_K)+2K ( 1-  [(1-e^{-1})e^{-2\lambda}]^{M(K)} )^{\beta/\alpha }.
\eeq 
In the next step we will choose a suitable $K$ to make \eqref{failphase2} sufficiently small and to achieve our third objective.

\mn
{\bf Phase 3: $[T, 3T]$.} In this phase we will build a walls in $[K,2K]$ that lasts fro times $[T,3T]$ and another one in $[-2K,-K]$ to keep $[-K,K]$ vacant during $[T,3T]$. Let $G_3$ denote the event that there are two such walls.

\mn
The density of $-1$'s at time $T$ is at least $\nu=3\pi(-1)/8$. 
The probability that a wall of $-1$ will not flip to 0 within time $2T =2 (t_0+\beta)/\alpha$ is $\exp(-2\theta(t_0+\beta))$. Since the flips at each site are independent, the probability that we will obtain a wall in $[K,2K] \times [T,3T]$ and a wall in $[-2K,-K]\times [T,3T]$ is 
\beq\label{success3}
P(G_3)\geq 1- 2\exp(-2K\nu \theta(t_0+\beta)),
\eeq
which is arbitrarily close to 1 if $K$ is sufficiently large. For a given $\ep>0$, we will choose $K$ large enough so that $\eqref{success3}\geq 1-\ep/2$. Having completed the choice of $K$, we note that if $K$, $\beta$, and $\varepsilon >0$ are fixed then for $\alpha$ sufficiently small, the probability $\eqref{failphase2}$ that we fail to kill all 1's in phase 2 is less than $\ep/2$.

Thus, with high probability all the 1's in $[-2K,2K]$ die by time $T$, and there are walls in $[K,2K] \times [T,3T]$ and
$[-2K,-K] \times [T,3T]$ to keep 1's from being reintroduced into $[-K,K] \times [T,3T]$.  Finally, combining the error probabilities in each phase, we have 
$$
P(\text{$(0,0)$ is closed})\leq P(G_2^c)+P(G_3^c)\leq \ep,
$$

At this point we have completed the proofs of  Lemma \ref{d1perc}. The desired conclusion follows using the reasoning at the end of Section \ref{sec:introth4}.

\clearp

\section{Proof of Theorem \ref{d2die}} \label{sec:pfth5}

\subsection{The block construction event has high probability}

Recall from \eqref{pid=2} that  $\pi(-1)=1/(1+\theta+4\lambda\theta)$  in $d=2$.

\begin{lemma}\label{gblock}
Let $p_0<p_c^{site}(\ZZ^2)$ be the constant such that 
\beq
-4\chi(p_0)^2\log((1-e^{-1})e^{-4\lambda})=1.
\label{defp0}
\eeq 
When $1-\pi(-1)<p_0$,  there are constants $c_T$ and $c_K>0$ so that if $T=c_T/\alpha$ and $K=T/c_K$ then 
$$
\lim_{\alpha\to0} P(\text{the block $[-2K,2K]^2\times [0,3T]$ is good})=1.
$$
\end{lemma}

\begin{proof}
As in the previous section $T=(t_0+\beta)/\alpha$ for some $t_0$ and $\beta$ to be determined. We need $K$ and $T$ to be comparable so that range of dependence between events in the block construction stays bounded as $T\to\infty$ . Again there are three phases in the construction but this time the proof is simpler if we start with the last one.

\mn
{\bf Phase 3: $[T, 3T]$.} Starting with no 1's in $[-2K,2K]^2$ at time $T$, we will upper bound the probability of $B$, the event there is a path of length $K$ starting at a point on the boundary of $[-2K,2K]^2$ along which the sum of the times between infections is at most $2T$.

\mn
Starting from all 1's on the boundary of $[-2K,2K]^2$,  it takes at least $K$ births to reintroduce a 1 into the region $[-K,K]^2$. There are $16K$ sites on the boundary of $[-2K,2K]^2$. Starting from a given site, the number of paths of length $K$ is $4^K$. Let $\{e_i: i\geq 1\}$ be independent exponential random variables with rate $\lambda$ and let $S_K = e_1 + \cdots e_K$. We need $\{S_K\leq 2T\}$ to reach $[-K,K]^2$ by time $2T$. Hence,
\beq
P(B) \le 16 K \cdot 4^K \cdot P(S_K \le 2T).
\label{eq31}
\eeq
Standard large deviations results (see, e.g., Section 2.7 in \cite{PTE5}) imply that if $0<c<1$ then we have 
\beq
P( S_K \le cK/\lambda) \le \exp(-\gamma(c) K)  \quad \text{ for some }\gamma(c)>0.
\label{eq32}
\eeq
 If $c_0$ is chosen so that $\exp(-\gamma(c_0)) = 1/5$, and we choose $c_K$ such that 
\beq
2T=2c_K K \le (c_0/\lambda)K
\label{cond1}
\eeq 
then combining \eqref{eq31} and \eqref{eq32} shows that 
\beq\label{G3}
P(B)\leq  16K \cdot \left( \frac{4}{5}\right)^K\to 0 \quad \text{ as }K\to\infty.
\eeq

\mn
{\bf Phase 1: $[0, t_0/\alpha]$.} In this phase we build a giant component of $-1$'s in which the holes are small.

\mn
As in Section \ref{sec:pfth4} we consider a comparison process $\hat{\xi}_t$ with the following transition rates:
\begin{center}
\begin{tabular}{clccl}
$1 \to 0$ & at rate 1 & \qquad & $0\to -1$ & at rate $\alpha$  \\
$0 \to 1$ & at rate $ 4\lambda$ & \qquad & $-1\to 0$ &  at rate $\theta\alpha$ 
\end{tabular}
\end{center}
In the original process $\xi_t$, the transition rates are the same except that 0 turns to 1 at rate $\lambda N_1$, where $N_1$ is the number of occupied neighbors. In $d=2$, this rate is always $\leq 4\lambda$. Hence if $\hat\xi_0(x)=1$ then $P(\hat{\xi}_t(x)=-1) \le  P(\xi_t(x)=-1)$. 

Using an argument in Phase 1 in the previous section, we can pick a time $t_0$ so that at time $t_0/\alpha$, the density of $-1$'s is $\geq (1-\delta)\pi(-1)$. If $\beta$ is chosen so that $e^{-\theta\beta} > (1-\delta)$ the density of $-1$'s that were alive at time $t_0/\alpha$ and persist to time $T=(t_0+\beta)/\alpha$ is at least $(1-\delta)^2 \pi(-1)$. If $1-\pi(-1) < p_0$ then when $\delta$ is small $p=1-(1-\delta)^2\pi(-1) < p_0$. We will call sites occupied by $-1$'s that persist from $t_0/\alpha$ to $(t_0+\beta)/\alpha$ \textit{closed}. All other sites are said to be \textit{open}.

\mn
{\bf Bounding the largest hole in the cluster of $-1$'s}

\mn
Let $p_0$ be the constand defined in \eqref{defp0}. Consider independent site percolation in which sites open with probability $p< p_0$.  Let $\mathcal{C}_0$ denote the open cluster containing the origin. It follows from (6.77) in \cite{Grimmett} that
$$P_p( |{\cal C}_0| \ge n) \le  2\exp\left( -\frac{n}{2\chi(p)^2} \right) \quad \text{if } n>\chi(p)^2,$$
where $\chi(p)$ is the mean cluster size when the open probability is $p$. Theorem 6.108 in \cite{Grimmett} states that $\chi(p)$ is an analytic function of $p$ on $[0,p_c)$, implying that $\lim_{p\to 0} \chi(p)=1$.

Let $A_K$ denote the event that there exists an open cluster of size at least $n(K)$ that overlaps with $[-2K,2K]^2$, where $n(K) = 4(1+\delta)\chi(p)^2 \log K$. We have
\beq\label{probAK}
P_p(A_K)\leq 16K^2 \cdot 2\exp\left( -\frac{n(K)}{2\chi(p)^2} \right)\to 0 \quad \text{as }K\to\infty.
\eeq

\mn
{\bf Phase 2: $[t_0/\alpha, (t_0+\beta)/\alpha]$.} In this phase we will kill all 1's in $[-2K,2K]^2$.

\mn
As in the previous section the probability to kill a particular particle in time 1 is at least $  q_\lambda = (1-e^{-1})e^{-4\lambda}$.  In the contact process on a finite set of size $n\le n(K)$, the process dies out by time 1 with probability at least
$$
q_{\lambda}^{n(K)}=K^{4(1+\delta)\chi(p)^2\log q_\lambda}\equiv K^{-r}
$$
where $r = -4(1+\delta)\chi(p)^2\log q_\lambda$. Hence the probability that the contact process on a set of size at most $n(K)$ has not died out by time $\beta/\alpha$ is less than
$$ 
(1-K^{-r})^{\beta/\alpha}.$$

Let $G_2$ be the event that all the 1's in $[-2K,2K]^2$ die during this phase. We have 
\beq\label{csurv}
P(G_2^c)\leq P_p(A_K)+16K^2 \cdot (1-K^{-r})^{\beta/\alpha}.
\eeq
By the definition of $p_0$, we have $-4\chi(p)^2\log q_\lambda<1$ for $p<p_0$. We can choose $\delta>0$ sufficiently small so that $r=-4(1+\delta)\chi(p)^2\log q_\lambda<1$.
Note $\beta/\alpha=\beta T/(t_0+\beta)$. If $T=c_K K$ for some $c_K>0$, since $r<1$ it follows from \eqref{probAK} and \eqref{csurv}  that 
\beq\label{G2}
P(G_2^c) \leq P_p(A_K)+16K^2 \cdot (1-K^{-r})
^{c_K\beta K/(t_0+\beta)}\to 0 \quad \text{ as }K\to \infty.
\eeq
 This completes the proof of Lemma \ref{gblock}. 
\end{proof}

\subsection{Comparison with oriented percolation}\label{sec:compop}

First we introduce some notation used in \cite{CDP13}. Let $D=d+1$, where we assume $d=2$. Let $\AA$ be a $D\times D$ matrix satisfying the following conditions: (i) if $x=(x_1,\dots,x_D)$ has $x_1+\cdots +x_D=1$ then the $D$-th coordinate of $\AA x$, denoted by $(\AA x)_D$, satisfies $(\AA x)_D=1$,  and (ii) if $x$ and $y$ are orthogonal then so are $\AA x$ and $\AA y$. Let $\QQ=\{ \AA x: x\in[-1/2,1/2]^D\}$ and $\LL_D=\{ \AA x: x\in\ZZ^D\}$, so that the collection $\{z+\QQ: z\in \LL_D\}$ is a tiling of space by rotated cubes. 

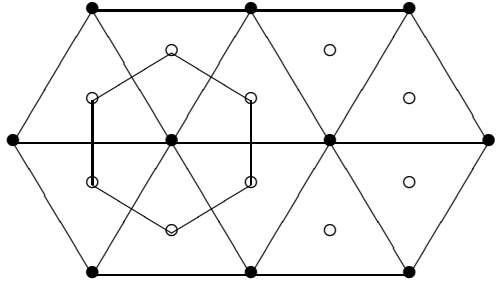
\begin{figure}[h]
\begin{center}
\begin{picture}(220,140)
\put(17,68){$\bullet$}
\put(47,118){$\bullet$}
\put(47,18){$\bullet$}
\put(77,68){$\bullet$}
\put(107,118){$\bullet$}
\put(107,18){$\bullet$}
\put(137,68){$\bullet$}
\put(167,118){$\bullet$}
\put(167,18){$\bullet$}
\put(197,68){$\bullet$}
\put(20,70){\line(1,0){180}}
\put(50,20){\line(1,0){120}}
\put(50,120){\line(1,0){120}}
\put(50,20){\line(3,5){60}}
\put(50,120){\line(3,-5){60}}
\put(20,70){\line(3,5){30}}
\put(20,70){\line(3,-5){30}}
\put(170,20){\line(-3,5){60}}
\put(170,120){\line(-3,-5){60}}
\put(200,70){\line(-3,-5){30}}
\put(200,70){\line(-3,5){30}}
\put(47,84){$\circ$}
\put(47,52){$\circ$}
\put(107,84){$\circ$}
\put(107,52){$\circ$}
\put(77,102){$\circ$}
\put(77,34){$\circ$}
\put(167,84){$\circ$}
\put(167,52){$\circ$}
\put(137,102){$\circ$}
\put(137,34){$\circ$}
\put(50,86){\line(0,-1){32}}
\put(110,86){\line(0,-1){32}}
\put(50,86){\line(5,3){30}}
\put(110,86){\line(-5,3){30}}
\put(50,54){\line(5,-3){30}}
\put(110,54){\line(-5,-3){30}}
\end{picture}
\caption{$\HH'_k$ (black dots) and $\HH'_{k-1}$ (white dots), which are the corners of the Voronoi region containing $x$}
\label{fig:L4}
\end{center}
\end{figure}

Let $\HH_k=\{z\in \LL_D: z_D=k\}$ be the points on ``level" $k$. We will often write elements of $\HH_k$ in the form $(z,k)$ where $z\in \RR^d$. Let $\HH'_k=\{ z\in \RR^d: (z,k)\in \HH_k\}$. Let $\{e_1,\dots, e_D\}$ be the standard basis in $\RR^D$ and put $v_i=\AA e_i$ for $i=1,\dots, D$. Writing $v_i=(v'_i,1)$, $v'_i\in \RR^d$ has length $\sqrt{D-1}$. The definition for $\HH'_k$ implies that $\HH'_{k+1}=\{ v'_i+x: x\in \HH'_k, 1\leq i\leq D\}$.

For $x\in \HH'_k$, let $\VV_x$ be the Voronoi region for $x$, i.e., the closed set of points in $\RR^d$ that are closer to $x$ in Euclidean norm than to all the other points of $\HH'_k$ (including ties). See Figure \ref{fig:L4} for an illustration of Voronoi region. It follows from the definition of Voronoi region that for every $k\in \mathbb{N}$,
$$\cup_{x\in \HH'_k}\VV_x=\RR^d.$$
We can further note that $\VV_x$ is contained in the closed ball of radius $D$ centered at $x$ (see (5.1) in \cite{CDP13}). Thus for any $L>0$, if $c_L=L/(2D)$ then 
\beq
c_L \VV_x \subset c_L x +[-L,L]^d \quad \text{ and so }\cup_{x\in \HH'_k} c_L x+[-L,L]^d=\RR^d.
\eeq
Each $(z,n)\in \HH_n$ is associated with a block $(c_L z+ [-2K,2K]^2)\times  [ nT, (n+3)T]$. We extend the definition of good blocks from $(0,0)$ to $z\in \LL_D$ by translation, and say $(z,n)$ is open if the associated block is a good block, i.e.,
$$\text{ if }(z,n)\in W^0_n, \text{then }(c_L z+[-K,K]^2)\times [ (n+1)T, (n+3) T] \text{ contains no 1's}.$$
For our purpose we will choose $L=K$ so that there is no hole in the dead zone. That is, if all sites in $\HH_n$ are good then $\cup_{z\in \HH'_n} c_Lz+[-K,K]^2=\RR^2$.

\medskip
Next we give a description of the oriented percolation process. It will be constructed from the set of random variables $\{\eta(z), z\in \LL_D\}$, where $\eta(z)\in\{0,1\}$. If $\eta(z)=1$ then the site $z$ is said to be open, otherwise it is closed. We have seen in Section \ref{d1comparison} that an $M$-dependent oriented percolation with open probability $1-\ep$ dominates an oriented percolation where each site is open with independent probability $1-f_M(\ep)$ such that $\lim_{\ep\downarrow 0} f_M(\ep)=0$. Hence for the rest of this section, we can suppose $\{\eta(z), z\in \LL_D\}$ are i.i.d. with $P(\eta(z)=1)=1-\theta$ for some sufficiently small $\theta>0$.

The edge set $\EE_\uparrow$ for $\LL_D$ is defined to be the set of all oriented edges from $z$ to $z+v_i$, $z\in \LL_D$, $1\leq i\leq D$. A sequence of points $(z_0,\dots,z_n)$ in $\LL_D$ is called an \textit{open path} from $z_0$ to $z_n$ if there is an edge in $\EE_{\uparrow}$ from $z_i$ to $z_{i+1}$ and $z_i$ is open for $i=0,\dots, n-1$. We write $z_0\to z_n$ if there exists an open path from $z_0$ to $z_n$. Given the initial wet sites $W_0\subset \HH_0$, we say $z\in \HH_n$ is wet if $z_0\to z$ for some $z_0\in W_0$. Let $W^0_n$ be the set of wet sites in $\HH_n$ when $W_0=\{0\}$. Let $\Omega^0_\infty=\{ W^0_n\neq \emptyset \text{ for all }n\geq 0\}$. Let $\bar{W}_n$ be the set of wet sites in $\HH_n$ when all the sites in $\HH_0$ are wet. Call sites in $\bar V_n=\HH_n \backslash \bar W_n$ \textit{dry}. The connection between $W^0_n$ and $\bar W_n$ is made in Lemma 5.1 in \cite{CDP13}.
\begin{lemma}[Lemma 5.1 in  \cite{CDP13}]\label{lemma5.1}
Let $H^r_n=\{ (z,n)\in \LL_D: z\in [-r,r]^d\}$. There are $\theta_1>0$ and $r_1>0$ such that if $\theta<\theta_1$ and $r\leq r_1$ then as $N\to\infty$,
$$P(\Omega^0_\infty \text{ and }W^0_n\cap \HH^{rn}_n \neq \bar W_n \cap \HH^{rn}_n \text{ for some }n\geq N) \to 0.$$
\end{lemma}

 Previously in the case $d=1$, we prove extinction of our process $\xi_t$ on the event $\Omega^0$.
This is not good enough in $d=2$ because the space-time blocks associated with sites in $V^0_n=\HH_n \backslash W^0_n$ might contain 1's. To prove extinction, we need to trace backwards in time by looking at the dual process of $\xi_t$. However, in the corresponding grid $\LL_D$, 1's may spread sideways through several dry regions, so we need to introduce an additional set of edges $\EE_\downarrow$ for $\LL_D$. Let 
$\EE_{\downarrow}$ be the set of oriented edges from $z$ to $z-v_i$, $1\leq i\leq D$, and from $z$ to $z+v_i-v_j$ for $1\leq i \neq j \leq D$.

 Lemma \ref{lemma5.1} implies that on $\Omega^0_\infty$, $W^0_n\cap \HH^{rn}_n=\bar W_n\cap \HH^{rn}_n$ for large $n$. Hence we will consider the dry sites. Now fix $r>0$ and let $\BB_n$ be the dry sites in $\HH^{rn/4}_n$ connected to the complement of $\cup_{m=n/2}^n \HH^{rm/2}_m$ by a path of dry sites on the graph with edges $\EE_\downarrow$, where the last site in the path need not be dry. Lemma 5.5 in \cite{CDP13} will give our desired result.

\begin{lemma}[Lemma 5.5 in \cite{CDP13}]\label{lemma5.5}
There exists some $\theta_0>0$ so that if $\theta\leq \theta_0$ then 
$$P( \BB_n\neq \emptyset \text{ infinitely often})=0.$$
\end{lemma}

Recall that a wet site in $\HH^{rn}_n$ corresponds to a good space-time block while a dry site in $\HH^{rn}_n$ corresponds to a space-time block which may contain a 1 in its translation of $[-K,K]^2\times [T,3T]$. Since a dry site $(z,n)\in \HH^{rn/4}_n$ corresponds to a block containing a 1 in $(c_L z+ [-K,K]^2)\times [(n+1)T, (n+3)T]$, there must be a dual path leading from this 1 to a site outside of $\cup_{m=n/2}^n \HH^{rm/2}_m$, which corresponds to a path of dry sites in $\EE_\downarrow$. This can not happen on $\{\BB_n=\emptyset\}$. 

Lemma \ref{gblock} guarantees that when $\alpha$ is sufficiently small we have $\theta\leq \theta_0$. It then follows from Lemma \ref{lemma5.5} that for sufficiently large $n$, there cannot be any 1's in $[-c_Lrn/4,c_Lrn/4]^2\times [(n+1)T,(n+3)T]$. Otherwise there must be a dry site in $\HH^{rn/4}$ connected to the component of  $\cup_{m=n/2}^n \HH^{rm/2}_m$ by a path of dry sites through edges in $\EE_\downarrow$, which contradicts the fact that $\BB_n=\emptyset$ for sufficiently large $n$. Therefore, we have a linearly growing dead zone $[-c_Lrn/4,c_Lrn/4]^2$ that will take over the whole space, leading to the extinction of the 1's.

\clearp

\end{document}